\documentclass[a4paper,12pt]{amsart}

\usepackage{amssymb,amscd,amsthm,mathrsfs}
\usepackage[ansinew]{inputenc}
\usepackage[centertags]{amsmath}
\usepackage{epsfig,graphicx}

 \def\RR{{\mathbb R}}  \def\TT{{\mathbb T}}
   
 \def\ZZ{{\mathbb Z}}

\def\cA{\mathcal{A}}    
   \def\cN{\mathcal{N}}

    \def\cW{\mathcal{W}}
\def\cF{\mathcal{F}}  \def\cL{\mathcal{L}}

\newtheorem*{teo*}{Theorem}
\newtheorem*{prop*}{Proposition}
\newtheorem*{cor*}{Corollary}
\newtheorem*{goal*}{Goal}
\newtheorem*{mainteo}{Main Theorem}
\newtheorem*{mainteoformal}{Main Theorem (Precise Statement)}
\newtheorem*{teoA}{Theorem A}
\newtheorem*{teoB}{Theorem B}

\newtheorem{teo}{Theorem}[section]
\newtheorem*{conj}{Conjecture}

\newtheorem{cor}[teo]{Corollary}

\newtheorem{lemma}[teo]{Lemma}
\newtheorem{prop}[teo]{Proposition}

\newcommand{\bi}{\begin{itemize}}
\newcommand{\ei}{\end{itemize}}

\theoremstyle{definition}

\theoremstyle{remark}

\newcommand{\Heis}{\mathcal{H}}

\newcommand{\bbR}{\mathbb{R}}
\newcommand{\bbQ}{\mathbb{Q}}
\newcommand{\bbZ}{\mathbb{Z}}
\newcommand{\bbN}{\mathbb{N}}
\newcommand{\bbT}{\mathbb{T}}

\newcommand{\Es}{E^s}
\newcommand{\Ec}{E^c}
\newcommand{\Eu}{E^u}
\newcommand{\Ecu}{E^{cu}}
\newcommand{\Ecs}{E^{cs}}
\newcommand{\Ws}{\cW^s}

\newcommand{\Wu}{\cW^u}

\newcommand{\inv}{^{-1}}

\newcommand{\Fc}{\mathcal{F}^c}

\newcommand{\Fcu}{\mathcal{F}^{cu}}
\newcommand{\Fcs}{\mathcal{F}^{cs}}

\newcommand{\Fsig}{\mathcal{F}^{\sigma}}
\newcommand{\As}{\mathcal{A}^s}
\newcommand{\Ac}{\mathcal{A}^c}
\newcommand{\Au}{\mathcal{A}^u}
\newcommand{\Acu}{\mathcal{A}^{cu}}
\newcommand{\Acs}{\mathcal{A}^{cs}}

\newcommand{\Asig}{\mathcal{A}^{\sigma}}

\newcommand{\Lcs}{L^{cs}}
\newcommand{\Lcu}{L^{cu}}
\newcommand{\cLcs}{\cL^{cs}}
\newcommand{\cLcu}{\cL^{cu}}

\newcommand{\tM}{\tilde M_A}
\newcommand{\HD}{d_H}

\newcommand{\ep}{\epsilon}
\newcommand{\lam}{\lambda}
\newcommand{\Lam}{\Lambda}
\newcommand{\gam}{\gamma}
\newcommand{\Gam}{\Gamma}

\newcommand{\heis}{\mathfrak{h}}

\newcommand{\spacearrow}{\quad\Rightarrow\quad}

\newcommand{\dist}{\operatorname{dist}}
\newcommand{\length}{\operatorname{length}}
\newcommand{\volume}{\operatorname{volume}}

\newcommand{\interior}{\operatorname{int}}
\newcommand{\Aff}{\operatorname{Aff}}

\theoremstyle{remark}
\newtheorem*{remark} {\bf Remark}

\newtheorem{assumption}[teo] {\bf Assumption}

\newcommand{\eps}{\varepsilon}

\author[A. Hammerlindl]{Andy Hammerlindl}
\address{School of Mathematical Sciences, Monash University,
Victoria 3800 Australia} \email{andy.hammerlindl@monash.edu}

\author[R. Potrie]{Rafael Potrie}
\address{CMAT, Facultad de Ciencias, Universidad de la Rep\'ublica, Uruguay}
\urladdr{www.cmat.edu.uy/$\sim$rpotrie}\email{rpotrie@cmat.edu.uy}

\title[Partial hyperbolicity in 3-solvmanifolds]{Classification of partially hyperbolic diffeomorphisms in 3-manifolds with solvable fundamental group}
\thanks{A.H.~was partially supported by
the Australian Research Council under Grant DP120104514.
R.P.~was partially supported by CSIC group 618, FCE-3-2011-1-6749 and the
Palis--Balzan project.}

\begin{document}
\begin{abstract}
A classification of partially hyperbolic diffeomorphisms on 3-dimensional
manifolds with (virtually) solvable fundamental group is obtained.
If such a diffeomorphism does not admit a periodic attracting or repelling
two-dimensional torus, it is dynamically coherent and leaf conjugate to a
known algebraic example.
This classification includes manifolds which support Anosov flows, and it
confirms conjectures by Rodriguez Hertz--Rodriguez Hertz--Ures
(\cite{HHU}) and Pujals (\cite{BW}) in the specific case of solvable
fundamental group.
\bigskip


%\medskip

\end{abstract}

\maketitle

\section{Introduction}\label{SectionIntroduction}
\subsection{Partial hyperbolicity in dimension three}

Partial hyperbolicity has been widely studied in recent years not only as a generalization of uniform hyperbolicity, but also to
%try to
characterize robust dynamical behaviour in terms of geometric structures
invariant under the tangent map. Partially hyperbolic diffeomorphisms play a
central role in both the study of stable ergodicity and robust transitivity
(see \cite{BDV,WilkinsonSurvey}).

The goal of this paper is to contribute to the classification problem of
partially hyperbolic diffeomorphisms in dimension three.
A diffeomorphism $f$ of a 3-manifold $M$ is partially hyperbolic if there is a
splitting of the tangent bundle into three $Df$-invariant one-dimensional
subbundles
$TM= E^s \oplus E^c \oplus E^u$
such that $E^s$ is contracted by $Df$, $E^u$ is expanded by $Df$, and this dominates
any expansion or contraction on $E^c$.

The classification can be stated informally as follows:

\begin{mainteo}
Every partially hyperbolic diffeomorphism on a 3-manifold with (virtually)
solvable fundamental group is, up to finite lifts and iterates, one of the
following:
\begin{itemize}
    \item a Derived-from-Anosov system,
    \item a skew-product,
    \item a deformation of a suspension Anosov flow, or
    \item a non-dynamically coherent system containing \\ a periodic torus
        tangent either to $E^c \oplus E^u$ or $E^c \oplus E^s$.
\end{itemize}
\end{mainteo}

The next section gives precise statements both of the definition of partial
hyperbolicity and of the four families in the classification.

It is necessary to consider finite iterates and lifts to a finite cover
as otherwise the result would not hold (see \cite{BW}).
In Appendix \ref{Apendix-FiniteQuotients}, we present a complete
classification which takes into account all possible quotients and roots of
the above models.

All four families in the Main Theorem are non-empty.
For the first three, simple algebraic models exist.
Constructing examples in the fourth family, however, requires much
more work. The recent proof of the existence of such non-dynamically
coherent examples, given in \cite{HHU}, came
as a surprise to the partially hyperbolic community.
Since a periodic torus must be an attractor if transverse to $E^s$ or
a repeller if transverse to $E^u$, the existence of such objects can be ruled
out under mild additional assumptions.  For instance, there are no such tori
if the partially hyperbolic system is transitive (or even chain recurrent)
or has an invariant measure of full support.
Appendix \ref{Appendix-AnosovTori} further studies the dynamics of systems
containing these periodic tori.

The proof of the Main Theorem builds on previous results such as
\cite{BBI,BBI2,BI,HHU3,Hammerlindl,HNil,HP,Pw,Pot} and ideas developed in \cite{BW}.
In fact, \cite{HP} establishes the Main Theorem in the case of virtually
nilpotent fundamental group.

The contribution of this paper is to treat for the first time partially
hyperbolic diffeomorphisms on manifolds which support Anosov flows (namely
3-manifolds with sol geometry) without \emph{a priori} assumptions on the
existence of foliations tangent to the center direction.
For these manifolds, no previous results on dynamical coherence were known,
not even in the more restrictive case of absolute partial hyperbolicity.

We expect many of the ideas presented here to apply in more general
contexts in order to complete the classification of
partial hyperbolicity on all 3-manifolds.
In fact, much of the reasoning here relies only on two crucial properties
which hold in this case:

\begin{itemize}
\item[(i)] The mapping class group of the manifold is finite; every
    diffeomorphism has an iterate which is isotopic to the identity.
\item[(ii)] Foliations of the manifold can be accurately described; if
    there are no torus leaves, no two leaves on the universal cover have
    bounded Hausdorff distance.
\end{itemize}

Using these two properties, we are able
to show that all the leaves of certain \emph{branching
foliations} introduced in \cite{BI} are fixed by the dynamics on the
universal cover.
This is already a big step in comparing such partially hyperbolic
diffeomorphisms to Anosov flows as these flows always fix the leaves of their
invariant weak foliations.
This property of fixed leaves is used to classify Anosov and expansive flows
on a variety of 3-manifolds % pun intended
(see \cite{Brunella} and the references therein),
and such results might be generalized to the partially hyperbolic setting.

\subsection{Precise definitions and statement of results}

As this paper deals only with manifolds in dimension three, $M$ will always
denote a compact connected $3$-manifold without boundary.
A diffeomorphism $f: M \to M$ is
\emph{partially hyperbolic} if there is $N>0$ and a $Df$-invariant continuous
splitting $TM =E^s \oplus E^c \oplus E^u$ into one-dimensional
subbundles such that for every $x\in M$:
\[ \|Df^N|_{E^s(x)}\| < \|Df^N|_{E^c(x)}\| < \|Df^N|_{E^u(x)}\|
\quad\text{and}\quad
 \|Df^N|_{E^s(x)}\| < 1 < \|Df^N|_{E^u(x)}\|. \]
This definition is sometimes called \emph{strong} partial hyperbolicity,
as opposed to \emph{weak} partially hyperbolicity, where there is only a
splitting into two subbundles.
(See \cite[Appendix B]{BDV} for more discussion.)
There also exists an \emph{absolute} version of partial hyperbolicity
considered in Appendix \ref{Appendix-AnosovTori} of this paper.

A \emph{foliation} is a partition of $M$ into $C^1$-injectively immersed
submanifolds tangent to a continuous distribution (see
\cite{CandelConlon}).

For all partially hyperbolic systems, the bundles $E^s$ and $E^u$ integrate
into invariant \emph{stable} and \emph{unstable} foliations $\cW^s$ and
$\cW^u$ (see \cite{HPS}). The integrability of the other bundles does not hold
in general.
The system is \emph{dynamically coherent} if there exist
invariant foliations $\cW^{cs}$ and $\cW^{cu}$ tangent to $E^{cs}=E^s\oplus
E^c$ and $E^{cu}=E^c \oplus E^u$ respectively. This automatically implies the
existence of an invariant foliation $\cW^c$ tangent to $E^c$. See \cite{BuW2}
for more information on dynamical coherence and reasons for considering this
definition.

As mentioned earlier, there are non-dynamically coherent examples which
contain normally hyperbolic periodic two-dimensional tori. % (\cite{HHU}).
The discoverers of these examples conjectured that
such tori are the unique obstruction to dynamical coherence in dimension three:

\begin{conj}[Rodriguez Hertz--Rodriguez Hertz--Ures (2009) \cite{HHU}]
Let $f: M \to M$ be a partially hyperbolic diffeomorphism such
that there is no periodic two-dimensional torus tangent to
$E^{cs}$ or $E^{cu}$. Then, $f$ is dynamically coherent.
\end{conj}

We prove this conjecture in the case of virtually solvable fundamental group.

\begin{teoA} Let $f$ be a partially hyperbolic diffeomorphism of a
    3-manifold $M$ such that $\pi_1(M)$ is virtually solvable.
    If there is no periodic torus tangent to
    $E^{cs}$ or $E^{cu}$,
    then there are unique invariant foliations tangent
    to $E^{cs}$ and $E^{cu}$. In particular, $f$ is dynamically
    coherent.
\end{teoA}

Once dynamical coherence is established, there is a natural way to classify
partially hyperbolic diffeomorphisms from a topological point of view. This
uses the notion of leaf conjugacy first introduced in \cite{HPS} (see also
\cite{BW}).
Two partially hyperbolic diffeomorphisms $f,g : M
\to M$ are
\emph{leaf conjugate} if they are both dynamically coherent and there is a
homeomorphism $h: M \to M$ which maps every center leaf $L$ of $f$ to a center
leaf $h(L)$ of $g$ and such that $h(f(L))=g(h(L))$.

In 2001, E.~Pujals stated a conjecture (later formalized in \cite{BW}) that in
dimension three all transitive partially hyperbolic diffeomorphisms belong to
one of three already known classes.

\begin{conj}[Pujals (2001), Bonatti--Wilkinson (2004)]
    Let $f$ be a transitive partially hyperbolic diffeomorphism
    on a 3-manifold.
    Then, modulo finite lifts and iterates, $f$ is
    leaf conjugate to one of the following models:
\begin{itemize}
    \item A linear Anosov diffeomorphism of $\TT^3$.
    \item A skew product over a linear Anosov automorphism of $\TT^2$.
    \item The time-one map of an Anosov flow.
\end{itemize}
\end{conj}

Recently, Rodriguez Hertz, Rodriguez Hertz, and Ures stated a
slightly different version of this conjecture where ``transitive'' is
replaced by ``dynamically coherent.''

Our Main Theorem is then a proof of the conjecture in the case of
virtually solvable fundamental group.

\begin{mainteoformal}
Let $f$ be a partially hyperbolic
diffeomorphism on a 3-manifold with virtually solvable fundamental group.
If there is no periodic torus tangent to $E^{cs}$ or $E^{cu}$, then,
modulo finite lifts and iterates, $f$ is leaf conjugate to one of the
following models:
\begin{itemize}
    \item A linear Anosov diffeomorphism of $\TT^3$.
    \item A skew product over a linear Anosov automorphism of $\TT^2$.
    \item The time-one map of a suspension Anosov flow.
\end{itemize}
\end{mainteoformal}

In the first case, the Anosov diffeomorphism $A:\bbT^3\to\bbT^3$ will have
three distinct real eigenvalues and therefore be partially hyperbolic.
In the second case, a \emph{skew product} is a bundle map
$f:M\to M$ where $M$ is a circle bundle over $\bbT^2$ defined by a continuous
bundle projection $P:M\to\bbT^2$ and an Anosov base map $A:\bbT^2\to\bbT^2$
such that $Pf=AP$.
This family includes the standard skew products on
$\bbT^3 = \bbT^2 \times \bbT^1$, as well as partially hyperbolic maps on other
3-dimensional nilmanifolds.  See \cite{HP} for further discussion on these two
cases.

This paper mainly treats the third case. In this paper, a
\emph{suspension manifold} is taken to mean a manifold of the form
$M_A = \bbT^2 \times \bbR / {\sim}$ where $(Ax, t) \sim (x, t+1)$
and $A : \bbT^2 \to \bbT^2$ is a hyperbolic toral automorphism. A
\emph{suspension Anosov flow} is a flow $\varphi_t(x,s)=(x,s+t)$
on a suspension manifold $M_A$.
% I think suspension manifold should be used over solvmanifold when we only
% mean $M_A$.  Tori and nilmanifolds are also first-class solvmanifolds in
% their own right!   --Andy
%
%Sometimes the term
%\emph{solvmanifold} will be used to refer to the manifolds $M_A$
%for some hyperbolic matrix $A$.

Previous classification results used either the fact that the action in
homology was partially hyperbolic (\cite{BI,Hammerlindl,HNil,HP}) or that
many center leaves were fixed by $f$ (\cite{BW}). In fact, based on many
previous results, it was proved in \cite{HP} that the previous conjectures
hold under the assumption that the fundamental group of $M$ is virtually
nilpotent (which is equivalent to having polynomial growth of volume).

In this paper, we treat a family of manifolds
which admit Anosov flows and show that every partially hyperbolic system on
these manifolds is comparable to a flow.

\begin{teoB} Let $f: M \to M$ be a partially hyperbolic dynamically coherent diffeomorphism on a manifold $M$ such that $\pi_1(M)$ is virtually solvable but not virtually nilpotent. Then, there exists a finite iterate of $f$ which is leaf conjugate to the time-one map of a suspension Anosov flow.
\end{teoB}

To end this introduction, we present an algebraic version of the Main Theorem
which follows directly from the results presented in Appendix
\ref{Apendix-FiniteQuotients}.
For a Lie group $G$, an \emph{affine map} is a diffeomorphism of the form
$G \to G,\  g \mapsto \Phi(g) \cdot g_0$
where $\Phi:G \to G$ is a Lie group automorphism and $g_0 \in G$.
Let $\Aff(G)$ denote the set of affine maps on $G$.

\begin{teo*}
    If $f$ is a partially hyperbolic, dynamically coherent diffeomorphism of a
    3-manifold $M$
    with (virtually) solvable fundamental group, then there are
    $G$, a solvable Lie group, $\Gamma < \Aff(G)$, and $\alpha \in \Aff(G)$
    such that $G / \Gamma$ is a manifold diffeomorphic to $M$,
    $\alpha$ descends to a partially hyperbolic map
    $\alpha_0 : G / \Gamma \to G / \Gamma$
    and $f$ and $\alpha_0$ are leaf conjugate.
\end{teo*}

\subsection{Main steps of the proof  and organization of the paper}

The Main Theorem has already been proven in the case of a virtually
nilpotent fundamental group \cite{HP}.  This paper therefore proves the
theorem only in the case that the group is virtually solvable and not
virtually nilpotent.  This breaks into five major steps.

\begin{prop} \label{proplist1}
    If $f$ is a partially hyperbolic diffeomorphism of a 3-manifold $M$ with a
    fundamental group which is virtually solvable and not virtually nilpotent,
    then there is a finite normal covering of $M$ by a suspension manifold
    $M_A$.
\end{prop}

\begin{prop} \label{proplist2}
    If $M$ is finitely covered by a suspension manifold and $f:M \to M$
    is partially hyperbolic, then $M$ is diffeomorphic to a suspension
    manifold.
\end{prop}

\begin{prop} \label{proplist3}
    For any homeomorphism $f$ of a suspension manifold $M_A$, there is $n  \ge  1$
    such that $f^n$ is homotopic to the identity.
\end{prop}

\begin{prop} \label{proplist4}
    If $f$ is a partially hyperbolic diffeomorphism on a suspension manifold
    $M_A$ and there is no periodic torus tangent to
    $\Ecu$ or $\Ecs$, then there are unique invariant foliations tangent to
    $\Ecu$ and $\Ecs$.
\end{prop}

\begin{prop} \label{proplist5}
    If $f$ is a partially hyperbolic diffeomorphism on a suspension manifold
    $M_A$ such that
    $f$ is dynamically coherent and homotopic to the identity,
    then $f$ is leaf conjugate to the time-one map of a suspension Anosov
    flow.
\end{prop}

Together these propositions also imply Theorems A and B.  (In the case of
Theorem A, one must also refer to \cite{HP} for the virtually nilpotent case.)
Thus, the remaining goal of this paper is to prove each of the five
propositions.

Section \ref{SectionReductionSolvCase} proves Proposition \ref{proplist1}
based on results given in
\cite{Pw}.
The section also reduces Proposition 1.4 to an equivalent result (Lemma
\ref{lemmadc}) which is easier to prove.

Fundamental to the proofs of Propositions \ref{proplist4} and \ref{proplist5} are 
the branching foliations $\Fcs$ and $\Fcu$ developed by Burago and Ivanov
\cite{BI}.
Section \ref{SectionBranchingFoliations} introduces these objects and their
properties.

Section \ref{Section-SolvAndModel} details the geometry of the suspension
manifold $M_A$, its universal
cover $\tM$, and defines two model foliations $\Acs$ and $\Acu$ on this cover.

Section \ref{Section-ModelCorrect} compares the branching foliations to the
model foliations showing that each leaf of $\Fcs$ is a finite distance from a
leaf of $\Acs$ and vice versa, and that the same holds for $\Fcu$ and
$\Acu$.  %This is a key step in the overall proof.

Section \ref{Sect-Coherence} shows dynamical coherence, completing the proof
of Proposition \ref{proplist4}.

Section \ref{Sect-LeafConj} establishes the leaf conjugacy, Proposition
\ref{proplist5}.

Appendix \ref{Apendix-FiniteQuotients} proves Propositions \ref{proplist2} and \ref{proplist3} and thus finishes the proof of the
Main Theorem.  This appendix also gives an algebraic classification of
partially hyperbolic diffeomorphisms on suspension manifolds which are
not homotopic to the identity and classifies the possible quotients of
partially hyperbolic diffeomorphisms on the 3-torus and other 3-nilmanifolds.

Appendix \ref{Appendix-AnosovTori} considers the case where the branching
foliation $\Fcs$ or $\Fcu$ has a torus leaf.  This implies the existence of a
periodic torus, and the appendix rules out such tori for absolutely partially
hyperbolic systems.

\smallskip
{\bf Acknowledgements: }\textit{We have benefited from discussions with
C.~Bonatti, J.~Brum, S.~Crovisier, E.~Pujals, J.A.G.~Roberts, M.~Sambarino,
and A.~Wilkinson.}

%%%%%%%%%%%%%%%%%%%%%%%%%%%%%%%%%%%%%%%%%%%%%%%%%%%%%%%%%%%%%%%%%%%%%%%%%%%%%%%%%%%%%%%%%%%%%%%%%%%%%

\section{Reduction to the case of a suspension manifold}\label{SectionReductionSolvCase}

We start by proving Proposition \ref{proplist1}.

\begin{proof}
    [Proof of Proposition \ref{proplist1}]
    Parwani showed that if $f$ is partially hyperbolic on a 3-manifold $M$
    with virtually solvable fundamental group, then $M$ is finitely covered by
    a torus bundle over a circle \cite[Theorem 1.10]{Pw}.
    That is, $M$ is finitely covered by a suspension manifold
    $M_A$ with the caveat that the toral automorphism $A$ is not necessarily
    hyperbolic.
    %If $A$ is an automorphism of $\bbZ^2$,
    Then, the fundamental group of $M_A$ verifies $\pi_1(M_A) \cong \bbZ^2 \rtimes_A
    \bbZ$, and if $A$ is not hyperbolic, a direct computation shows that the
    group must be nilpotent
    implying that the fundamental group of $M$ is virtually nilpotent.
    As this is not possible by hypothesis,
    $A$ must be hyperbolic. %One deduces that $A$ is hyperbolic as desired.
%    However, if $A$ is not hyperbolic, then,
%    viewing $A$ as an automorphism of $\bbZ^2$,
%    the fundamental group of $M_A$ verifies $\pi_1(M_A) \cong \bbZ^2 \rtimes_A \bbZ$ and since $A$ is not hyperbolic, it is virtually nilpotent.
\end{proof}

\begin{remark}
    The exact statement in \cite[Corollary 1.11]{Pw} requires that the fundamental group
    has a \emph{normal} finite-index solvable subgroup
    $K \triangleleft \pi_1(M)$, instead of just a
    finite-index solvable subgroup $H < \pi_1(M)$.
    These two conditions are, in fact, equivalent.
    Any finite-index subgroup $H < G$ contains a normal finite-index
    subgroup $K \triangleleft G$.  Indeed, there is a natural homomorphism
    $g \mapsto (a H \mapsto g a H)$ from $G$ to the (finite) group of
    permutations of the cosets of $H$, and we may take $K$ to be the kernel.
    If $H$ is solvable, the subgroup $K < H$ is also solvable.
\end{remark}

%Now we present some results which will allow us to consider finite lifts and iterates and still get the results on dynamical coherence stated in Theorem A for the original diffeomorphism.

Section \ref{Sect-Coherence} gives the proof of the following lemma which adds
two additional assumptions to Proposition \ref{proplist4}.

\begin{lemma} \label{lemmadc}
    If $f$ is a partially hyperbolic diffeomorphism on a suspension manifold
    $M_A$ such that
    \begin{itemize}
        \item
        there is no $f$-periodic torus tangent to $\Ecu$ or $\Ecs$,
        \item
        $\Eu$, $\Ec$, $\Es$ are oriented and $Df$ preserves these orientations,
        and
        \item
        $f$ is homotopic to the identity,  \end{itemize}
    then there are unique $f$-invariant foliations tangent to
    $\Ecu$ and $\Ecs$.
\end{lemma}

To see how Proposition \ref{proplist4} follows from this lemma, we need two
results involving finite covers.

\begin{prop} \label{canlift}
    If $f:M \to M$ is a homeomorphism of a manifold and $p:N \to M$ is a finite
    cover, then some iterate of $f$ lifts to $N$.
    To be precise, there is $n  \ge  1$ and $g: N \to N$ such that
    $p \circ g = f^n \circ p$.
\end{prop}
\begin{proof}
    In order to lift a map $h: M \to M$ to $N$ by the cover $p$
    one must verify that $h_\ast p_\ast (\pi_1(N)) \subset
    p_\ast(\pi_1(N))$. Since $p$ is a finite cover, the subgroup $p_\ast(\pi_1(N))$
    has finite index in $\pi_1(M)$. Let $k \geq 1$ be this index.

    Given a group $G$ and $k  \ge  1$, there are finitely many subgroups of $G$ of
    index $k$ \cite[Theorem IV.4.7]{LyndonSchupp}. Therefore,
    there exists $n\ge 1$ such that $(f^n)_\ast (p_\ast(\pi_1(N)))
    = p_\ast(\pi_1(N))$ and one can lift $f^n$ to $N$ as
    desired.
    %This can be proved from the properties of covering spaces, lifts, and
%    fundamental groups, using the following algebraic fact:
%    given a group $G$ and $n  \ge  1$, there are finitely many subgroups of $G$ of
%    index $n$ \cite[Theorem IV.4.7]{LyndonSchupp}.
\end{proof}
\begin{prop} \label{liftuniq}
    Suppose
    \begin{itemize}
        \item
        $f:M \to M$ is a diffeomorphism,
        \item
        $E_M \subset TM$ is a $Df$-invariant subbundle,
        \item
        $p: N \to M$ is a finite normal cover,
        \item
        $E_N \subset TN$ is the lift of $E_M$, and
        \item
        $g: N \to N$ is a diffeomorphism such that $p \circ g = f^n \circ p$
        for some $n  \ge  1$.  \end{itemize}
    If there is a unique $g$-invariant foliation $\cW_N$ tangent to $E_N$, then
    there is also a unique $f$-invariant foliation $\cW_M$ tangent to $E_M$.
\end{prop}
\begin{proof}
    First, suppose there is a unique $f^n$-invariant foliation $\cW$ tangent
    to $E_M$.
    Then, $f(\cW)$ is also an
    $f^n$-invariant foliation and by uniqueness $f(\cW)=\cW$.
    That is, $\cW$ is an $f$-invariant foliation and since any $f$-invariant
    foliation is also $f^n$-invariant, it is the unique one.
    This shows that we may freely replace $f$ and $g$ by their iterates in the
    rest of the proof.

    Let $\Gamma$ be the finite group of deck transformations $\gam:N \to N$
    associated to the cover $p:N \to M$.
    Note that $\gam \mapsto g \inv \gam g$ is an automorphism of $\Gamma$.
    As $\Gamma$ is finite, there is $m$ such that
    $\gam = g^{-m} \circ \gam \circ g^m$ for all $m$.
    By replacing $f$ and $g$ by iterates, assume that $m = 1$ and
    $p \circ g = f \circ p$.

    If $\cW_N$ is the unique $g$-invariant foliation tangent to $E_N$,
    $\gam \circ g = g \circ \gam$ implies that $\gam(\cW_N)$ is also $g$-invariant, and as
    $E_N$ was lifted from a subbundle of $TM$, $\gam(\cW_N)$ is also tangent to
    $E_N$.  By uniqueness, $\gam(\cW_N)=\cW_N$ for all $\gam \in \Gam$, so $\cW_N$
    projects down to a foliation $\cW_M$ on $M$.
    As $p \circ g = f \circ p$ and locally $p$ is a diffeomorphism, it follows
    that $\cW_M$ is $f$-invariant.
    Uniqueness of $\cW_N$ implies uniqueness of $\cW_M$.
\end{proof}
We now prove Proposition \ref{proplist4} assuming Lemma \ref{lemmadc}
and Proposition \ref{proplist3}.

\begin{proof}
    [Proof of Proposition \ref{proplist4}.]
    There is a finite cover $p : M_1 \to M_A$ such that when lifted to $M_1$
    the bundles $\Eu$, $\Ec$, $\Es$ are orientable.
    One can show that $M_1$ is again a suspension manifold.
    By Lemma \ref{canlift}, there is a map $f_1:M_1 \to M_1$ such that
    $p \circ f_1 = f^k \circ p$.
    By Proposition \ref{proplist3}, there is $n$ such that $f_1^n$ is homotopic
    to the identity.
    Then $g := f_1^{2n}$ is also homotopic to the identity and preserves the
    orientations of the bundles.
    Note that if there is a $g$-periodic 2-torus tangent to $\Ecu$ or $\Ecs$,
    it projects to an $f$-periodic surface tangent to $\Ecu$ or $\Ecs$ which must
    again be a torus.
    As such a torus is ruled out by assumption, the conditions of Lemma
    \ref{lemmadc} hold and there are unique $g$-invariant foliations on $M_1$
    tangent to $\Ecu$ and $\Ecs$.
    Since $p \circ g = f^{2 n k} \circ p$, Lemma \ref{liftuniq} implies that there are
    unique $f$-invariant foliations on $M_A$ tangent to $\Ecu$ and $\Ecs$.
\end{proof}

%%%%%%%%%%%%%%%%%%%%%%%%%%%%%%%%%%%%%%%%%%%%%%%%%%%%%%%%%%%%%%%%%%%%%%%%%%%%%%%%%%%%%%%%%%%%%%%%%%%

\section{Branching foliations}\label{SectionBranchingFoliations}

The starting point for many recent results in partial hyperbolicity
is the masterful construction by Burago and Ivanov of branching foliations
associated to the dynamics \cite{BI}.
Using Novikov's compact leaf theorem and other deep properties in the foliation
theory of 3-manifolds, Burago and Ivanov proved a number of results which may
be stated as follows.

%Burago and Ivanov proved the existence of branching foliations for
%partially hyperbolic systems in dimension three \cite{BI}.
%The following theorem is the starting point of all the results on dynamical coherence. Its proof is quite involved and difficult and it also depends on deep theorems of the theory of foliations of 3-manifolds such as Novikov's Theorem \cite{CandelConlon}.
%

\begin{teo}
    [Burago-Ivanov \cite{BI}] \label{thmbran}
    Let $f_0:M \to M$ be a partially hyperbolic diffeomorphism on a 3-manifold
    such that $\Eu$, $\Ec$, and $\Es$ are orientable and the derivative $Df_0$
    preserves these orientations.
    Let $f:\tilde M \to \tilde M$ be a lift of $f_0$ to the universal cover.

    Then, there is a collection $\Fcs$ of properly embedded planes in $\tilde M$
    such that for every $L \in \Fcs${:}
    \begin{enumerate}
        \item
        $L$ is tangent to $\Ecs$,
        \item
        $f(L) \in \Fcs$,
        \item
        $\gam(L) \in \Fcs$ for all deck transformations $\gam \in \pi_1(M)$,
        \item \label{twospaces}
        $L$ cuts $\tilde M$ into two closed half-spaces $L^+$ and $L^-$, and
        \item \label{nocross}
        every $L' \in \Fcs$ lies either in $L^+$ or $L^-$ (or both if
        $L=L'$).  \end{enumerate}
    Finally, for every $x \in \tilde M$ there is at least one $L \in \Fcs$ such that
    $x \in L$.
\end{teo}
An analogous collection $\Fcu$ also exists. For the purposes of this paper, a \emph{branching foliation} will be a collection as in the statement of this theorem. The planes of the collection will be called \emph{leaves} of the branching foliation in analogy with true foliations.
Item (\ref{nocross}) of the list is equivalent to saying that leaves of the
branching foliation do not topologically cross.

\begin{assumption}
    For the rest of this section, assume $f_0$, $f$, $M$, $\tilde M$, and $\Fcs$ are
    as in Theorem \ref{thmbran}.
\end{assumption}

\begin{prop}
    \label{csuint}
    Each leaf of $\Fcs$ intersects an unstable leaf in at most
    one point.
\end{prop}
\begin{proof}
    On the universal cover $\tilde M$, one may choose an orientation of $\Eu$
    such that for a leaf $L \in \Fcs$ and at all points $x \in L$
    the orientation of $\Eu(x)$ is away from $L^-$ and into $L^+$.
    Therefore, any oriented path
    tangent to $\Eu$ passes from $L^-$ into $L^+$ when intersecting $L$, and
    a path can do this at most once.
\end{proof}
%This can be proved by considering the orientation of $\Eu$ and
%using item \ref{twospaces} of Theorem \ref{thmbran} since given $L
%\in \Fcs$ once an unstable curve intersects $L$ it stays in
%$L^+$.

The following proposition,
which follows from \cite[Remark 1.16]{BW}, will be used to
establish dynamical coherence.

\begin{prop} \label{disjointleaves}
    If every $x \in \tilde M$ is contained in exactly one leaf of $\Fcs$,
    then $\Fcs$ is a true foliation.
\end{prop}
Even when $\Fcs$ is not a true foliation, it is close to a true foliation in a
certain sense.  To state this precisely, we introduce a few
more definitions.

In a metric space $X$ where $d(\cdot,\cdot)$ is the distance function
and $Y$ and $Z$ are non-empty subsets,
define
$\dist(x,Z) := \inf_{z \in Z} d(x,z)$ and
Hausdorff distance as
\[
    \HD(Y,Z) := \max\{ \sup_{y \in Y} \dist(y,Z),
    \sup_{z \in Z} \dist(z, Y) \}.
\]
The Hausdorff distance may be infinite in general.

Suppose $\cF_1$ and $\cF_2$ are collections of subsets of a metric space $X$.
Then, $\cF_1$ and $\cF_2$ are \emph{almost parallel} if there is
$R > 0$ such that
\begin{itemize}
    \item
    for each $L_1 \in \cF_1$, there is $L_2 \in \cF_2$ with
    $\HD(L_1, L_2) < R$, and
    \item
    for each $L_2 \in \cF_2$, there is $L_1 \in \cF_1$ with
    $\HD(L_2, L_1) < R$.  \end{itemize}
In our case, these collections will either be foliations or branching
foliations and $X$ will be the universal cover of a manifold.

\begin{prop} \label{propbi72}
    The branching foliation $\Fcs$ is almost parallel to a foliation $\cF$ of $\tilde M$ by planes and this
    foliation quotients down to a foliation $\cF_0$ on $M$.
    Moreover, $\cF_0$ contains a compact surface $S$ if and only if there is $L \in
    \Fcs$ which quotients down to a compact surface on $M$ homeomorphic to $S$.
\end{prop}

This follows as a consequence of \cite[Theorem 7.2]{BI}.

The existence of a compact surface coming from the branching foliations
has strong dynamical consequences.

\begin{prop} \label{propcompactleaf}
    If $\Fcs$ has a leaf $L$ which projects to a compact surface on $M$,
    then there is a normally repelling $f_0$-periodic torus tangent to $\Ecs$.
\end{prop}

This is proved in Appendix \ref{Appendix-AnosovTori}, where partially hyperbolic systems with such
tori are studied in further detail.

In the specific case of 3-manifolds with solvable fundamental group, the
foliations without compact leaves have been classified.  This was done by
Plante for $C^2$ foliations where the
classification is up to equivalence \cite{Plante}.
For the foliations considered
in this paper, the classification only holds up to the notion of two
foliations being almost parallel when lifted to the universal cover
\cite{HP}.
In the case of suspension manifolds $M_A$, any such foliation, after lifting,
is almost parallel to one of two model foliations, $\Acs$ or $\Acu$, which will
be defined in the next section.

\begin{prop} \label{propmodel}
    If $\cF_0$ is a foliation of $M_A$ without compact leaves,
    then the lifted foliation $\cF$ on $\tM$ is almost parallel either to $\Acs$
    or $\Acu$.
\end{prop}

This is a restatement of \cite[Theorem B.6]{HP}.

%%%%%%%%%%%%%%%%%%%%%%%%%%%%%%%%%%%%%%%%%%%%%%%%%%%%%%%%%%%%%%%%%%%%%%%%%%%%%%%%%%%%%%

\section{Solvmanifolds and model foliations}\label{Section-SolvAndModel}

Suppose $A$ is a hyperbolic $2 \times 2$ integer matrix with integer entries
and determinant plus-minus one.
Then, there is $\lam>1$ such that $A$ has an unstable eigenspace in $\bbR^2$ with
associated eigenvalue $\pm \lam$, and a
stable eigenspace with associated eigenvalue $\pm \lam \inv$.

Consider a manifold $\tM$ diffeomorphic to $\bbR^2 \times \bbR$, but with a
non-Euclidean geometry to be defined shortly.
Each point in $\tM$ can be written as $(v,t)$ with $v \in \bbR^2$ and $t \in \bbR$.
Define foliations $\Acs$ and $\Acu$ of dimension two by{:}
\begin{itemize}
    \item
    $(v_1,t_1) \in \Acs((v_2, t_2))$
    if $v_1-v_2$ is in the stable eigenspace of $A$,
    \item
    $(v_1,t_1) \in \Acu((v_2, t_2))$
    if $v_1-v_2$ is in the unstable eigenspace of $A$,
\end{itemize}
Define foliations $\As$ and $\Au$ of dimension one 
by{:}
\begin{itemize}
    \item
    $(v_1,t_1) \in \As((v_2, t_2))$
    if $t_1=t_2$ and
    $(v_1,t_1) \in \Acs((v_2, t_2))$,
    \item
    $(v_1,t_1) \in \Au((v_2, t_2))$
    if $t_1=t_2$ and
    $(v_1,t_1) \in \Acu((v_2, t_2))$.
\end{itemize}
\begin{figure}[t]
\begin{center}
\includegraphics{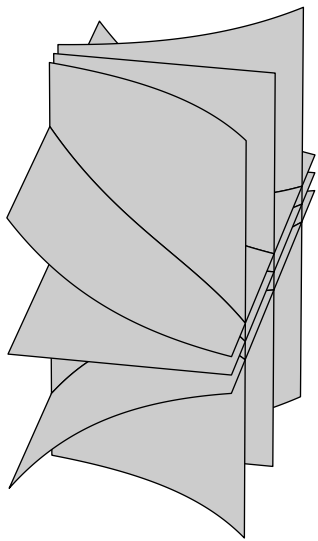}
\end{center}
\caption{A depiction of the model foliations $\Acs$ and $\Acu$.
    The planes separate apart in the left direction for one foliation and
    in the right direction for the other.
    Note that in the coordinate system for $\tM \approx \bbR^3$ introduced in
    Section \ref{Section-SolvAndModel},
    the foliations consist of geometric planes and it is the
    Riemannian metric which changes in $t$.
}
\end{figure}
Define a flow $\varphi:\tM \times \bbR \to \tM$ by $\varphi_t((v,s)) = (v,s+t)$
and let $\Ac$ be the foliation consisting of orbits of the flow.
Now, define a unique Riemannian metric on $\tM$ by requiring the
following properties{:}
\begin{itemize}
    \item
    at every point in $\tM$, an orthogonal basis of the tangent space can be
    formed by vectors in $\As$, $\Ac$, and $\Au$,
    \item
    $\varphi$ is a unit speed flow,
    \item
    if $\alpha$ is a simple curve in a leaf of $\Au$ with endpoints $(v_1,t)$
    and $(v_2,t)$, then $\length(\alpha) = \lam^t \|v_1-v_2\|$, and
    \item
    if $\alpha$ is a simple curve in a leaf of $\As$ with endpoints $(v_1,t)$
    and $(v_2,t)$, then $\length(\alpha) = \lam^{-t} \|v_1-v_2\|$.
\end{itemize}
In the last two items above, $\|v_1-v_2\|$ denotes the Euclidean distance in
$\bbR^2$.
Define functions $\gamma_i:\tM \to \tM$ by
\begin{align*}
    \gamma_1((v,t)) &= (v + (1,0), t), \\
    \gamma_2((v,t)) &= (v + (0,1), t), \\
    \gamma_3((A v,t)) &= (v, t+1).
\end{align*}
One can verify that these three functions are isometries which generate a
solvable group $G \subset \operatorname{Isom}(\tM)$ and that the quotient
$\tM / G$ is exactly the compact manifold $M_A$ defined earlier.
Further, the flow $\varphi$ quotients to an Anosov flow on $M_A$ and the
foliations $\cA^*$ quotient down to invariant foliations associated to the
flow.
Define $G_0 = \langle \gam_1, \gam_2 \rangle$.

% TODO: Make this an equation with label?
Define a function $p_1: \tM \to \bbR,\ (v,t) \mapsto t$.

\begin{lemma} \label{pconst}
    For all $\gam \in G$, one has $p_1(x) - p_1(y) = p_1(\gam(x)) - p_1(\gam(y))$.
    Further, $p_1 \circ \gam = p_1$ exactly when $\gam \in G_0$.
\end{lemma}
\begin{proof}
    This follows from the definitions of $p_1$ and $\gam_i$ ($i = 1,2,3$).
\end{proof}
\begin{lemma} \label{plip}
    For any $x,y \in \tM$, one has $|p_1(x) - p_1(y)|  \le  d(x,y)$.
\end{lemma}
\begin{proof}
    Consider a curve $\alpha:[a,b] \to \tM$ from $x$ to $y$ parameterized by arc
    length.  For each point $\alpha(t)$, there is an orthonormal basis
    $\{v^u,v^c,v^s\}$ of the tangent space such that the exterior derivative
    $dp_1$ satisfies $dp_1(v^u) = 0$, $dp_1(v^c) = 1$, and $dp_1(v^s) = 0$.
    This shows that
    \[
        \left| \frac{d}{dt} p_1(\alpha(t)) \right|  \le  1
    \]
    and therefore $|p_1(y) - p_1(x)| \leq b-a$.
\end{proof}
For any $x,y \in \tM$, there is a unique point of intersection
\[
    [x,y] := \Au(x) \cap \Acs(y).
\]
This point depends continuously on $x$ and $y$ and since the foliations are
invariant under $G$,
\[
    [\gam(x), \gam(y)] = \gam([x,y])
\]
for all $\gam \in G$.
Define $\dist_u(x, \Acs(y))$ as the length of the unstable segment between $x$
and $[x,y]$.
Note that
\[
    \dist_u(\varphi_t(x), \Acs(y)) = \lam^t \dist_u(x, \Acs(y))
\]
for all $x$, $y$, and $t$.

\begin{lemma} \label{RR1}
    For $R>0$ there is $R_1>0$ such that
    \[
        \dist(x,\cL) < R \spacearrow \dist_u(x,\cL) < R_1
    \]
    for all $x \in \tM$ and $\cL \in \Acs$.
\end{lemma}
\begin{proof}
    Define
    \[
        \Delta = \{ (x,y) \in \tM \times \tM : d(x,y)  \le  R \}.
    \]
    It is enough to show that $D:\Delta \to \bbR$ defined by
    $
        D(x,y) = \dist_u(x, \Acs(y))
    $
    is bounded.
    Define an action of $G$ on $\Delta$ by
    $\gam \cdot (x,y) = (\gam(x), \gam(y))$ and note that
    $D$ quotients to a continuous function $D_G:\Delta/G \to \bbR$.
    As $\Delta/G$ is compact, both $D_G$ and $D$ are bounded.

\end{proof}
\begin{lemma} \label{glueA}
    If $\{x_k\}$, $\{x'_k\}$ are sequences in leaves $\cL, \cL' \in \Acs$
    such that\\
    $\sup d(x_k, x_k') < \infty$ and $\lim_{k \to \infty} p_1(x_k) = +\infty$,
    then $\cL = \cL'$.
\end{lemma}
\begin{proof}
    By the last lemma, there is $R_1$ such that $\dist_u(x_k, \cL') < R_1$
    for all $k$.
    Let $t_k = -p_1(x_k)$.
    Then
    \[
        \lim_{k \to \infty} \dist_u(\varphi_{t_k}(x_k), \cL')  \le
        \lim_{k \to \infty} \lam^{t_k} R_1 = 0.
    \]
    This shows that there are points in $\cL \cap p_1 \inv(0)$ which come
    arbitrarily close to $\cL' \cap p_1 \inv(0)$.
    As the definition of $\Acs$ implies that its leaves intersect the
    plane $p_1 \inv(0)$ in parallel lines, the result follows.
\end{proof}
\begin{cor} \label{hdinfty}
    If $\cL, \cL' \in \Acs$ are distinct, then $\HD(\cL, \cL') = \infty$.
\end{cor}

A leaf of $\Acs$ splits $\tM$ into two pieces.  We will also need that an
arbitrarily large neighbourhood of a leaf splits $\tM$ into two pieces.

\begin{lemma} \label{twohalves}
    For any $R > 0$ and $\cL \in \Acs$, there is a partition of $\tM$ into three
    pieces $X, Y, Z$ such that
    \begin{itemize}
        \item
        $X$ and $Z$ are open and connected,
        \item
        $\HD(Y, \cL) < \infty$, and
        \item
        $\dist(y, \cL) < R$ implies $y \in Y$.
    \end{itemize}  \end{lemma}
\begin{proof}
    Let $Y = \{ y \in \tM : \dist_u(y, \cL)  \le  R_1 \}$
    where $R_1$ is given by Lemma \ref{RR1}.
    Let $X,Z$ be the two connected components of $\tM \setminus Y$.
    It is straightforward to verify the desired properties.
\end{proof}
\begin{lemma} \label{polygrow}
    For any $n>0$ the manifold with boundary defined by
    \[
        \tilde V_n := \{ x \in \tM : |p_1(x)|  \le  n \}
    \]
    has polynomial growth of volume.  That is,
    if $B_R(x)$ consists of all endpoints of paths in $\tilde V_n$ starting at
    $x$ and of length at most $R$, then the volume of $B_R(x)$ grows
    polynomially in $R$.
\end{lemma}
\begin{proof}
    The deck transformations in $G_0$ map $\tilde V_n$ to itself and the
    quotient $\hat V_n := \tilde V_n / G_0$ is compact and is therefore a
    finite volume manifold with boundary.  The result then follows from the
    fact that $G_0 \cong \bbZ^2$ has polynomial growth.
\end{proof}

\begin{lemma} \label{disjointgrow}
    If $X$ is a measurable subset of $\tM$ such that $X \cap \gam(X) =
    \varnothing$ for all $\gam \in G_0$ then
    the volume of
    \[
        X_n := \{ x \in X : |p_1(x)| < n \}
    \]
    grows linearly in $n$.
\end{lemma}
\begin{proof}
    Using that the isometry $\gam_3 \in G$ is volume
    preserving, it is not hard to show that the volume of $\hat V_n$ (as defined
    in the proof of the last lemma) is exactly
    $n$ times the volume of $\hat V_1$.
    Each $X_n \subset \tilde V_n$ projects to a set $\hat X_n \subset \hat V_n$.
    As this projection is one-to-one, the volumes of $X_n$ and $\hat X_n$ are
    the same. The result follows.
\end{proof}

%%%%%%%%%%%%%%%%%%%%%%%%%%%%%%%%%%%%%%%%%%%%%%%%%%%%%%%%%%%%%%%%%%%%%%%%%%

\section{Comparing foliations}\label{Section-ModelCorrect}

In this section, suppose $f_0$ is a partially hyperbolic diffeomorphism
which satisfies the hypotheses of Lemma \ref{lemmadc}. In particular, since it is homotopic to the identity, $f_0$ lifts to a diffeomorphism $f:\tM \to \tM$ which is a finite
distance from the identity.

In this section, the only dynamical system we consider is $f$.
We do not consider the flow $\varphi$ used in the last section.  Thus, the
words stable, unstable, and center here refer only to $f$.

As the hypotheses of Theorem \ref{thmbran} are satisfied, there
are branching foliations $\Fcs$ and $\Fcu$. By Propositions
\ref{propbi72}, \ref{propcompactleaf}, and \ref{propmodel},
the branching foliation $\Fcs$ is
almost parallel to either $\Acs$ or $\Acu$, defined in Section
\ref{Section-SolvAndModel}. Up to replacing $A$ with $A \inv$, we
may freely assume the first.

\begin{assumption}
    $\Fcs$ is almost parallel to $\Acs$ and throughout the rest of the paper $R>0$ is the
    associated constant in the definition.
\end{assumption}

From this assumption it will follow that $\Fcu$ is almost parallel to $\Acu$.
However, the proof of this is surprisingly involved and will occupy the rest of this section.  We must first prove a
number of properties relating $\Fcs$ to $\Acs$.

\begin{lemma} \label{uniqueCL}
    If $L \in \Fcs$, $\cL \in \Acs$, and $\HD(L, \cL) < \infty$, then
    \begin{enumerate}
        \item
        $\cL$ is the unique leaf in $\Acs$, such that $\HD(L, \cL) < \infty$,
        \item
        $\HD(L, \cL) < R$,
        \item
        $\HD(f^k(L), \cL) < R$ for all $k \in \bbZ$, and
        \item
        $\HD(\gam(L), \gam(\cL)) < R$ for all $\gam \in G$.
    \end{enumerate}  \end{lemma}
\begin{proof}
    The first item follows from Corollary \ref{hdinfty} and the fact that
    Hausdorff distance satisfies the triangle inequality.
    The second item follows from the first item and the definition of almost
    parallel.
    The third item then follows from the fact that $f$ is a bounded distance
    from the identity and therefore $\HD(f(X), X) < \infty$ for any subset
    $X \subset \tM$.
    The fourth item holds because $\gam \in G$ is an isometry
    which preserves both $\Fcs$ and $\Acs$.
\end{proof}
Lemma \ref{uniqueCL} shows there is a canonical function from leaves of $\Fcs$ to
leaves of $\Acs$.  This function is in fact a bijection.

\begin{lemma} \label{uniqueL}
    If $L \in \Fcs$, $\cL \in \Acs$, and $\HD(L, \cL) < \infty$, then
    $L$ is the unique leaf in $\Fcs$ such that $\HD(L, \cL) < \infty$,
\end{lemma}
To prove this, we first provide two sublemmas.

\begin{lemma} \label{Lamint}
    If for $\cL \in \Acs$ the union
    \[
        \Lam = \bigcup \{ L \in \Fcs : \HD(L, \cL) < \infty \}
    \]
    contains two distinct leaves of $\Fcs$,
    then this union has non-empty interior.
\end{lemma}
\begin{figure}[t]
\begin{center}
\includegraphics{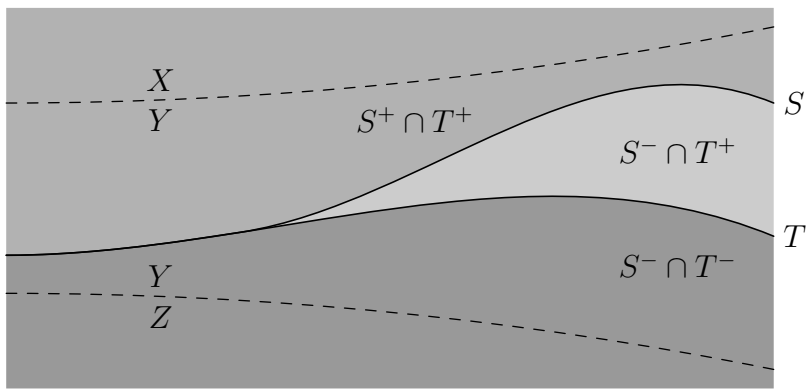}
\end{center}
\caption{A depiction of the half-spaces in the proof of Lemma \ref{Lamint}.
}
\end{figure}
\begin{proof}
    Let $X,Y,Z$ be as in Lemma \ref{twohalves}.
    Note by Lemma \ref{uniqueCL} that if $L \in \Fcs$ is contained in $Y$, then
    $L \subset \Lam$.

    Suppose $S,T \in \Fcs$ are distinct leaves, both contained in $\Lam$.
    As $S$ is a properly embedded surface, it splits $\tM$ into two closed
    half-spaces $S^+$ and $S^-$ such that
    \[
        \partial S^+ = \partial S^- = S^+ \cap S^- = S
        \quad \text{and} \quad
        \interior(S^+) \cup \interior(S^-) = \tM \setminus S.
    \]
    Since leaves of the branching foliation $\Fcs$ do not topologically cross,
    every leaf which intersects the interior of $S^+$ is contained in $S^+$
    and similarly for $S^-$.

    Since $S \subset Y$, $X$ must lie entirely in $S^+$ or $S^-$.
    Assume $X \subset S^+$.
    If $Z$ is also a subset of $S^+$, then $S^-$ would be contained in $Y$ and
    any leaf of $\Fcs$ which intersected the interior of $S^-$ would be in
    $\Lam$.
    As $S^-$ has non-empty interior, this would be enough to complete the
    proof.
    Hence, we may freely assume $Z$ is not a subset of $S^+$ and is instead
    a subset of $S^-$.
    Similarly, define $T^+$ and $T^-$ such that $X \subset T^+$ and $Z \subset
    T^-$.

    As $S  \ne  T$, there is $y \in T \setminus S$.
    Further, there is an open neighborhood $U$ containing $y$ which lies
    entirely in the interior of either $S^-$ or $S^+$ and intersects the
    interiors of both $T^+$ and $T^-$.
    This means that at least one of $S^+ \cap T^-$ and $S^- \cap T^+$
    has non-empty interior.
    As both of these intersections are unions of leaves of $\Fcs$ which lie
    entirely in $Y$, this shows that $\Lam$ has non-empty interior.
\end{proof}
For an unstable arc $J$, define $U^{cs}_1(J)$ as all points $x \in \tM$ such that
$\dist(x, J) < 1$ and every leaf of $\Fcs$ which contains $x$ also intersects
$J$.

\begin{lemma}\label{ucs1}
    There is $C > 0$ such that
    \[
        \volume(\interior(U^{cs}_1(J)))  \ge  C \length(J)
    \]
    for all unstable arcs $J$ of length at least one.
\end{lemma}
This is a slight modification of Lemma 3.3 of \cite{BBI2}.  For completeness,
we give a short proof.
\begin{proof}
    As the distributions $\Ecs$ and $\Eu$ are uniformly continuous and uniformly
    transverse, there is $\delta > 0$ such that the volume of the interior of
    $U^{cs}_1(J)$ is greater than $\delta$ for any unstable curve $J$ of length
    at least one.
    If $\length(J) > n$, there are $n$
    disjoint subcurves $J_1, \ldots, J_n \subset J$ of length one.
    If $U^{cs}_1(J_i)$ intersected $U^{cs}_1(J_j)$ at a point $x$, then
    any leaf of $\Fcs$ through $x$ would intersect $J$ at distinct points in
    $J_i$ and $J_j$ contradicting Proposition \ref{csuint}.
\end{proof}
\begin{proof}
    [Proof of Lemma \ref{uniqueL}.]
    If the claim is false, then by Lemma \ref{Lamint}, there is a leaf $\cL \in
    \Acs$ such that the corresponding union of leaves $\Lam$ has non-empty
    interior.

    If $\gam \in G_0$ and the interiors of $\Lam$ and $\gam(\Lam)$
    intersected, there would be a leaf $L \in \Fcs$ such that both
    $\HD(L, \cL) < \infty$ and $\HD(L, \gam(\cL)) < \infty$ and therefore
    $\cL = \gam(\cL)$ by Lemma \ref{uniqueCL}.
    One can check from the definitions of $G_0$ and $\Acs$ in the previous
    section that there is no such leaf with this property.
    Hence, by Lemma \ref{disjointgrow}, the volume of
    \[
        X_n := \{ x \in \interior(\Lam) : |p_1(x)| < n \}
    \]
    grows linearly in $n$.

    By Lemma \ref{uniqueCL}, we have $f(\Lam)=\Lam.$
    Take an unstable arc $J$ in the interior of $\Lam$.  As $f$ is a finite
    distance from the identity, there is $N$ such that for all $k \ge 1$,
    $f^k(J)$ is contained in a ball of radius $Nk$ and therefore
    $U^{cs}_1(f^k(J))$ is contained in a ball of radius $Nk+1$.
    By Lemma \ref{plip}, there is $C > 0$ such that $|p_1(x)| < Nk + C$
    for all $x \in U^{cs}_1(f^k(J))$.

    In other words, $\interior(U^{cs}_1(f^k(J))) \subset X_{Nk+C}$ for all $k \ge 1$.
    Since the length of $f^k(J)$ grows exponentially fast in $k$, and the
    volume of $X_{Nk+C}$ grows only linearly, Lemma \ref{ucs1} gives a
    contradiction.
\end{proof}

This leads to the following consequence which is of critical importance to the
overall proof.

\begin{cor} \label{fixleaf}
    If $L \in \Fcs$, then $f(L) = L$.
\end{cor}
\begin{proof}
    This follows from Lemmas \ref{uniqueL} and \ref{uniqueCL}.
\end{proof}

Now knowing that leaves of $\Fcs$ are fixed by $f$, we now proceed to prove
that $\Fcu$ is almost parallel to $\Acu$.
%Once the fact that leaves of the branching foliations are fixed is established
%we proceed to show that the foliations are close to the correct ones which is
%the main result of this section.

\begin{lemma} \label{glueF}
    If  $\{z_k\}$, $\{z'_k\}$ are sequences in leaves $L, L' \in \Fcs$
    such that\\
    $\sup d(z_k, z'_k) < \infty$
    and
    $\lim_{k \to \infty} p_1(z_k) = +\infty,$
    then $L = L'$.
\end{lemma}
\begin{proof}
    Let $\cL, \cL' \in \Acs$ be such that $\HD(L, \cL) < R$ and $\HD(L', \cL') < R$.
    Then, there are sequences $\{x_k\}$ and $\{x'_k\}$ in $\cL$ and $\cL'$
    such that
    $d(x_k, z_k) < R$  and $d(x'_k, z'_k) < R$
    for all $k$.
    By Lemma \ref{glueA} and the triangle inequality, $\cL=\cL'$.  By Lemma
    \ref{uniqueL}, $L=L'$.
\end{proof}
\begin{lemma} \label{rightbound}
    There is $K>0$ such that
    \[
        p_1(f^{-n}(x)) - p_1(x) < K
    \]
    for all $x \in \tM$ and $n  \ge  0$.
\end{lemma}
\begin{figure}[t]
\begin{center}
\includegraphics{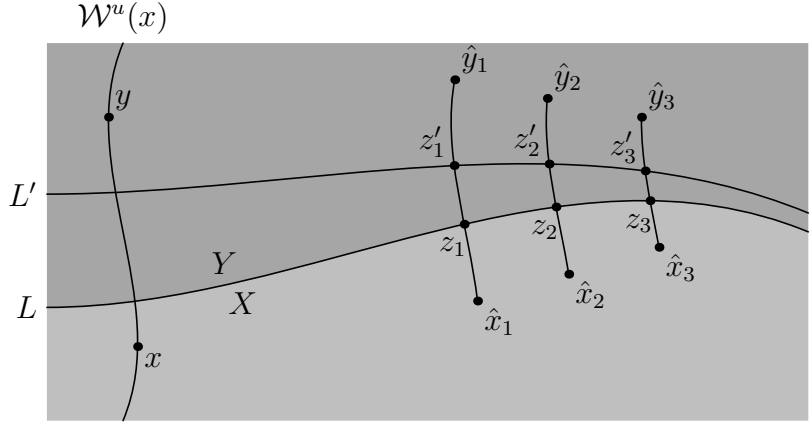}
\end{center}
\caption{A depiction of the sequences in the proof of Lemma \ref{rightbound}.
}
\end{figure}
\begin{proof}
    Suppose the statement is false.
    Then there are sequences $\{x_k\}$, $\{\hat x_k\}$ in $\tM$ and
    $\{n_k\}$ in $\bbN$ such that $\hat x_k = f^{-n_k}(x_k)$ for all $k$
    and $\lim_k p_1(\hat x_k) - p_1(x_k) = + \infty$.
    By Lemma \ref{pconst}, one may freely assume that ${x_k}$ lies within a
    compact fundamental domain.  By taking a subsequence, further assume that
    $x_k$ converges to a point $x$.
    Then,
    $\lim_k p_1(\hat x_k) = +\infty$.

    Define a convergent sequence $\{y_k\}$ such that each
    $y_k$ is connected to $x_k$ by an unstable segment of length exactly one.
    Then $y := \lim_k y_k$ is on the unstable leaf of $x$ and is distinct from
    $x$.
    Define $\hat y_k := f^{-n_k} (y_k)$.
    As $f \inv$ contracts unstable leaves,
    $\sup d(\hat x_k, \hat y_k) < \infty$.

    Choose a leaf $L \in \Fcs$ which intersects $\Wu(x)$ at a point strictly
    between $x$ and $y$.
    This leaf splits $\tM$ into closed half-spaces $X$ and $Y$
    such that $\partial X = \partial Y = X \cap Y = L$
    and $x \in \interior(X)$ and $y \in \interior(Y)$.

    For large $k$, $x_k \in \interior(X)$.
    By Lemma \ref{fixleaf},
    $\hat x_k$ is on the same leaf as $x_k$ and as leaves
    do not cross topologically, $\hat x_k \in X$.
    Similarly, $\hat y_k \in Y$ for large $k$.
    By taking subsequences, assume ``large $k$'' is all $k \ge 1$.

    As any path connecting $\hat x_k$ to $\hat y_k$ must intersect $L$,
    there is a sequence $\{z_k\}$ in $L$ such that
    \[
        \sup d(\hat x_k, z_k)  \le  \sup d(\hat x_k, \hat y_k) < \infty
    \]
    which also implies $\lim p_1(z_k) = +\infty$.
    Now choose a leaf $L' \in \Fcs$ which intersects the unstable segment
    between $x$ and $y$ at a point distinct from $L$.
    By the same construction, there is a sequence $\{z'_k\}$ in $L'$
    such that
    $\sup d(\hat x_k, z'_k) < \infty$
    and therefore
    $\sup d(z_k, z'_k) < \infty.$
    By Lemma \ref{glueF}, $L$ and $L'$ are the same leaf, which by Proposition
    \ref{csuint} gives a contradiction.
\end{proof}

We now have the tools necessary to compare $\Fcu$ and $\Acu$.
%Now it is possible to show that the branching foliation $\Fcu$ is almost parallel to the correct foliation.

\begin{prop}\label{Prop-AlmostParallelToCorrect}
    The branching foliation $\Fcu$ is almost parallel to $\Acu$.
\end{prop}
\begin{proof}
    $\Fcu$ is almost parallel either to $\Acu$ or $\Acs$.
    Assume the latter in order to derive a contradiction.
    Notice that a center-unstable branching foliation for a partially
    hyperbolic diffeomorphism is a center-stable branching foliation for
    its inverse.  That is, if $g = f \inv$, then $\Fcu_f = \Fcs_g$ and so the
    assumption is that $\Fcs_g$ is almost parallel to $\Acs$.
    Repeating all of the results of this section with $g$ in place of $f$, one
    proves a version of Lemma \ref{rightbound} for $g${:}
    \begin{quote}
        \em There is $K>0$ such that
        \[
            p_1(f^n(x)) - p_1(x) = p_1(g^{-n}(x)) - p_1(x) < K
        \]
        for all $x \in \tM$ and $n  \ge  0$.
    \end{quote}
    Taken with the original Lemma \ref{rightbound} for $f$, this implies that
    there is $K$ such that
    \[
        |p_1(f^n(x)) - p_1(x)| < K
    \]
    for all $x \in \tM$ and $n \in \bbZ$.

    Take an unstable curve $J$ such that $J \subset \tilde V_1$ in the
    notation of Lemma \ref{polygrow}.  The above inequality implies that
    $f^n(J) \subset \tilde V_{K+1}$ for all $n$ and therefore
    $U^{cs}_1(f^n(J)) \subset \tilde V_{K+2}$ (where $U^{cs}_1$ is defined just
    before Lemma \ref{ucs1}).
    Arguing as in the proof of Lemma \ref{uniqueL}, the volume of
    $U^{cs}_1(f^n(J))$ grows exponentially, but its diameter grows only
    linearly.  This contradicts the polynomial growth of volume given in Lemma
    \ref{polygrow}.
\end{proof}

%%%%%%%%%%%%%%%%%%%%%%%%%%%%%%%%%%%%%%%%%%%%%%%%%%%%%%%%%%%%%%%%%%%%%%%%%%%%%%%%%%%%%%%%%%%

\section{Dynamical coherence}\label{Sect-Coherence}

This section gives the proof of dynamical coherence. The assumptions are exactly the same as in the previous section. The starting point of the proof is the following consequence of Proposition \ref{Prop-AlmostParallelToCorrect}:

\begin{lemma} \label{Ncomp}
    There is $N>0$ such that for every $\Lcs \in \Fcs$ and $\Lcu \in \Fcu$,
    the intersection $\Lcs \cap \Lcu$ consists of at most $N$ connected
    components, each of which is a properly embedded line.
\end{lemma}
\begin{proof}
    Each component of $\Lcs \cap \Lcu$ is a one-dimensional manifold in $\Lcs$
    transverse to the unstable foliation $\Ws$ restricted to $\Lcs$.  Since no
    circle can be transverse to a foliation of a plane by lines, each
    component is a line.  % TODO: Add line/plane citation?
    If a component is not properly embedded, it accumulates at a point
    $y \in \Lcu$ and $\Wu(y)$ intersects $\Lcs$ more than once, contradicting
    Proposition \ref{csuint}.
    Distinct components $C_1$ and $C_2$ must be bounded away from each other
    by some uniform distance $\ep > 0$, as otherwise a single unstable leaf
    would intersect both $C_1$ and $C_2$ and again give a contradiction.

    By Lemma \ref{uniqueCL} (and its $cu$ counterpart) and Proposition \ref{Prop-AlmostParallelToCorrect}, there are unique leaves
    $\cLcs \in \Acs$ and $\cLcu \in \Acu$ such that
    $\HD(\Lcs, \cLcs) < R$
    and
    $\HD(\Lcu, \cLcu) < R$.
    For any $t \in \bbR$, define the disk
    \[
        D_t := \{ x \in \tM :
        \dist(x, \cLcs) < R,
        \ \dist(x, \cLcu) < R,
        \ p_1(x) = t \}.
    \]
    By Lemma \ref{RR1} (and its $cu$ counterpart), the diameter of $D_t$ is
    finite and depends purely on $R$ and not $\cLcu$, $\cLcs$, or $t$.
    In particular, there is $K > 0$ such that any subset of $D_t$ with at
    least
    $K$ points has two points at distance less than $\ep$.

    Since a connected component $C$ of $\Lcs \cap \Lcu$ is properly embedded,
    its image $p_1(C) \subset \bbR$ is of the form
    $[a, +\infty)$, $(-\infty, a]$, or $\bbR$.
    If there are more than $2K$ components, then there is $t \in \bbR$ such that
    $t \in p_1(C_i)$ for $K$ different components $C_1, \ldots, C_K$.
    Each $C_i$ then intersects $D_t$, and this gives a contradiction.
\end{proof}

The following key property in the proof is inspired by a similar argument due to Bonatti and Wilkinson (see Corollary 3.11 of \cite{BW}):

\begin{lemma} \label{noper}
    The diffeomorphism $f$ has no periodic points.
\end{lemma}
\begin{remark}
    Here, as in all of this section, this is a statement about the map
    $f$ on the universal cover $\tM$.
    A partially hyperbolic map on the compact manifold $M_A$ can certainly
    have periodic points.
\end{remark}
\begin{proof}
    Suppose $f^n(x) = x$.
    Take $y \in \Wu(x)$ and $L \in \Fcs$ such that $y \in L$.
    As $f(L)=L$ and $y$ is the unique intersection of $\Wu(x)$ and $L$, it
    follows that $f^n(y)=y$.
    This contradicts the fact that $f \inv$ must contract the unstable segment
    between $x$ and $y$.
\end{proof}
\begin{lemma} \label{toinfty}
    For all $x \in \tM$, $\lim_{n \to \infty} p_1(f^n(x)) = +\infty$.
\end{lemma}
\begin{proof}
    The point $x$ lies on a connected component $C$ of a set of the form
    $\Lcs \cap \Lcu$.
    By Lemma \ref{Ncomp}, there is $n>0$ such that $f^n(C)=C$ and by Lemma
    \ref{noper} this map has no fixed points.
    By the properties of fixed-point free homeomorphisms of the real line,
    the orbit $f^{n k}(x)$ must eventually leave any compact subset of $C$.
    Since for any $T > 0$, one can show that the set
    $\Lcs \cap \Lcu \cap p_1 \inv([-T,T])$ is compact,
    this implies that
    $
        \lim_{k \to \infty} | p_1(f^{n k}(x)) | = \infty
    $
    and from Lemma \ref{rightbound} the result follows.
\end{proof}
\begin{prop} \label{uniqfoln}
    The branching foliation $\Fcs$ is the unique $f$-invariant, $G$-invariant branching foliation
    tangent to $\Ecs$ and it is a true foliation.
\end{prop}
\begin{proof}
    Define a map $H: \tM \to \Acs$ by requiring that
    \[
        \sup_{n  \ge  0} \dist(f^n(x), H(x)) < \infty.
    \]
    Using Corollary \ref{fixleaf} and Lemmas \ref{uniqueCL}, \ref{toinfty}, and
    \ref{glueA}, one can check that $H$ is well-defined and its fibers $H
    \inv(\{\cL\})$ are exactly the leaves of $\Fcs$.
    Since the definition of $H$ does not use $\Fcs$, this shows there is only
    one such possible branching foliation.  As the fibers are disjoint,
    Proposition \ref{disjointleaves} shows that $\Fcs$ is a true foliation.
\end{proof}
Taken with the corresponding statement for the ${cu}$ foliation, Proposition
\ref{uniqfoln} implies dynamical coherence
and completes the proof of Lemma \ref{lemmadc}.
Modulo the proof of Proposition \ref{proplist3}, which is left to Appendix
\ref{Apendix-FiniteQuotients}, this also establishes Proposition
\ref{proplist4} and Theorem A.

%%%%%%%%%%%%%%%%%%%%%%%%%%%%%%%%%%%%%%%%%%%%%%%%%%%%%%%%%%%%%%%%%%%%%%%%%%%%%%%%%%%%%%%%%%%%

\section{Leaf conjugacy}\label{Sect-LeafConj}

This section gives the proof of Proposition \ref{proplist5}.
Assume $f_0:M_A \to M_A$ satisfies the hypotheses of
Proposition \ref{proplist5}.
As $f_0$ is homotopic to the identity, it lifts to a function $f:\tM \to \tM$
which is a finite distance from the identity.  As $f_0$ is dynamically
coherent, there are $f_0$-invariant (non-branching) foliations tangent to
$\Ecs$ and $\Ecu$ and, by a result announced in \cite{HHU}, these foliations do not
have compact leaves.  Therefore, Proposition \ref{propmodel} applies and these
foliations lift to foliations on $\tM$ which are almost parallel either to
$\Acs$ or $\Acu$.

The proofs in Sections \ref{Section-ModelCorrect} and \ref{Sect-Coherence}
now follow as before, with these true foliations taking the place of the
branching foliations.
There is, however, one small caveat.
It is possible for $f_0$ to be both dynamically coherent and to have an
invariant torus tangent to $\Ecs$ or $\Ecu$.
(See \cite{HHU} for an example.)
As such, there may be a branching foliation on $\tM$ with a leaf which
quotients down to a compact surface on $M_A$.
Such a branching foliation will not be almost parallel to $\Acs$ or to $\Acu$.
This issue only affects the proof of uniqueness in Proposition \ref{uniqfoln}.
Since Proposition \ref{uniqfoln} is not needed to establish the leaf
conjugacy, this is not an issue in proving Proposition \ref{proplist5}.

\begin{prop} \label{rightmove}
    There is $n_0>0$ such that $p_1(f^n(x)) > p_1(x)+1$
    for all $x \in \tM$ and $n>n_0$.
\end{prop}
\begin{proof}
    By Lemma \ref{pconst}, it is enough to prove this for a compact fundamental
    domain.
    By Lemma \ref{toinfty}, the sets
    \[
        U_n = \{ x \in \tM : p_1(f^n(x)) - p_1(x) > K+1 \}.
    \]
    for $n \in \bbN$ form an open cover of $\tM$.
    Thus, some finite union
    $
        U_{n_1} \cup \cdots \cup U_{n_k}
    $
    covers a fundamental domain.
    One can then show using Lemma \ref{rightbound} that $n_0 = \max\{n_1, \ldots,
    n_k\}$ satisfies the desired property.
\end{proof}

By adapting a proof of Fuller regarding flows \cite{fuller1965}, we prove the
following.

\begin{lemma} \label{projc}
    There is a continuous function $p_c:\tM \to \bbR$ such that
    \begin{itemize}
        \item
        for each leaf $L \in \Fc$, the restriction $p_c|_L : L \to \bbR$
        is a $C^1$ diffeomorphism,
        \item
        for each $t \in \bbR$, the level set $p_c \inv(\{t\})$ is an embedded $C^0$
        surface in $\tM$, and
        \item
        $p_c(\gamma(x)) - p_c(x) = p_1(\gamma(x)) - p_1(x)$
        for all $\gamma \in G$ and $x \in \tM$.
    \end{itemize}  \end{lemma}
\begin{proof}
    By Lemma \ref{Ncomp} and Proposition \ref{rightmove},
    there is $n>0$ such for all $x \in \tM$, $f^n(x)$ is on the same center leaf
    as $x$ and $p_1(f^n(x)) >  p_1(x) + 1$.
    By a compactness argument, the distance $d_c(x,f^n(x))$ along the center
    leaf is uniformly bounded both above and away from zero.
    Consequently, there is $T>0$ such that if $x$ and $y$ are on
    the same center leaf and $d_c(x,y) > T$, then $|p_1(x) - p_1(y)| > 1$.

    For a leaf $L \in \Fc$, let $\alpha:\bbR \to L$ be a path parameterizing $L$ by
    arc length and oriented such that $p_1 (\alpha(t)) \to +\infty$ as $t \to
    +\infty$.  Define $p_c$ on $L$ by
    \[
        p_c(\alpha(t)) = \frac{1}{T} \int_{t}^{t+T} p_1(\alpha(s))\, d s.
    \]
    By the Fundamental Theorem of Calculus,
    \[
        \frac{d}{dt}\ p_c(\alpha(t))
        = \frac{1}{T} \Big[p_1(\alpha(t+T)) - p_1(\alpha(t))\Big]
        > \frac{1}{T}
      \]
    so that $p_c \circ \alpha$ is a $C^1$ diffeomorphism.
    From this, the above listed properties are easily verified.
\end{proof}

Recall from Section \ref{Section-SolvAndModel} that points in $\tM$ can be written in
coordinates $(v,t)$ with $v \in \bbR^2$ and $t \in \bbR$
and that the group $G$ of deck transformations is generated
by the isometries $\gam_i:\tM\to\tM$ ($i=1,2,3$).
In these coordinates, $\gam_3 \inv(v,t) = (Av, t-1).$
Define a surface $\tilde S = p_c \inv(\{0\}) \subset \tM$ and a map
$\psi:\tilde S \to \tilde S$ by $\psi(x) = \gam_3 \inv(y)$
where $y$ is the unique point in $\Fc(x)$ such that $p_c(y) = 1$.

One can show that if $(v,t) \in \tilde S$, then $(v+z, t) \in \tilde S$ for
$z \in \bbZ^2$ and
\[
    \psi(v+z,t) = \psi(v,t) + (Av, 0).
\]
Now quotient $\psi$ down to a map $\psi_0:S \to S$ on $S \subset M_A$.
The surface $S$ is homeomorphic to a 2-torus and the induced action of of
$\psi_0$ on the
fundamental group of $S$ is given by the linear map $A$.

\begin{lemma} \label{expanse}
    The map $\psi$ is expansive.
\end{lemma}
\begin{proof}
    Suppose that for $\ep > 0$ there are points $x$ and $y$ such that
    $d(\psi^n(x), \psi^n(y)) < \ep$
    for all $n \in \bbZ$.
    Then,
    $d(\gam_3^n \psi^n(x), \gam_3^n \psi^n(y)) < \ep$
    and Lemma \ref{glueF} shows that $x$ and $y$ lie on the same leaf
    $\Lcs$ of $\Fcs$.
    An analogous argument shows that they lie on the same leaf
    of $\Fcu$.
    If $x$ and $y$ lie on distinct center leaves
    $L_x,L_y \in \Fc$
    and if $\ep$ is sufficiently small, then there is an unstable arc
    connecting $L_x$ to $L_y$.
    This arc intersects $\Lcs$ in distinct points, contradicting Lemma
    \ref{csuint}.
    Hence, $x$ and $y$ lie on the same center leaf, and since $\tilde S$
    intersects each center leaf exactly once, $x=y$.
\end{proof}
\begin{remark}
    Using Proposition \ref{rightmove} and Lemma \ref{expanse}, it is easy to
    show that $f$ is \emph{plaque expansive} as defined in \cite{HPS}.
\end{remark}
\begin{prop} \label{fconj}
    The quotient $f_0:M_A \to M_A$ of $f:\tM \to \tM$ is leaf conjugate to the
    time-one map of a suspension Anosov flow.
\end{prop}
\begin{proof}
    By a result of Lewowicz, any expansive homeomorphism of a torus is
    topologically conjugate to a linear Anosov diffeomorphism
    \cite{Lew}.
    The linear diffeomorphism is determined by the action on the fundamental
    group.

    In the case of $\psi$ and its quotient $\psi_0$, this implies that there is
    a homeomorphism $H: \tilde S \to \bbR^2$ so that $H \circ \psi = A \circ H$
    and $H(v+z,t) = H(v,t) + (Az, 0)$ for all $z \in \bbZ^2$.

    Now define a unique homeomorphism $h : \tM \to \tM$ by requiring
    \begin{enumerate}
        \item
        that the leaf of $\Fc$ through the point $x \in \tilde S$ is mapped to
        the leaf of $\Ac$ through the point $(H(x), 0) \in \tM$, and
        \item
        that $p_c \circ h = p_1.$
    \end{enumerate}
    By its construction, $h$ maps $\Fc$ to $\Ac$.
    Suppose $x,y \in \tilde S$ are such that $H(x)$ and $H(y)$ are on the same
    stable leaf of $A$.  By the conjugacy,
    $d(\psi^n(x), \psi^n(y)) \to 0$ as $n \to \infty$, which implies that
    $d(\gam_3^n \psi^n(x), \gam_3^n \psi^n(y)) \to 0$.
    As $\psi$ was defined in such a way that $\gam_3 \psi(x)$ is on the same
    center leaf as $x$, it follows from Lemma \ref{glueF} that $x$ and $y$ are
    on the same leaf of $\Fcs$.  Hence, $h(\Fcs)=\Acs$.  Similarly,
    $h(\Fcu)=\Acu$.

    Since each leaf of $\Acs$ intersects each leaf of $\Acu$ in a single leaf of
    $\Ac$, applying $h \inv$ shows that a leaf of $\Fcs$ intersects a leaf of
    $\Fcu$ in a single leaf of $\Fc$.
    Consequently, $f$ fixes all leaves of $\Fc$, and $h$ is a
    leaf conjugacy between $f$ and the time-one map $(v,t) \mapsto (v,t+1)$ of
    the lifted Anosov flow on $\tM$.

    This quotients down to a leaf conjugacy on the compact manifold $M_A$.
\end{proof}
This finishes the proof of leaf conjugacy for partially hyperbolic
diffeomorphisms which are homotopic to the identity (Proposition
\ref{proplist5}). Up to Proposition \ref{proplist3} which is proved in Appendix \ref{Apendix-FiniteQuotients} this completes the proof of the Main Theorem.

Further, the homeomorphism $h$ constructed in the last proof has several useful
properties listed here.

\begin{lemma} \label{niceh}
    There is a homeomorphism $h:\tM \to \tM$ such that
    \begin{itemize}
        \item
        $h(\cF^\sigma) = \cA^\sigma$ for $\sigma = c, cs, cu$,
        \item
        $h \circ \gamma = \gamma \circ h$ for all $\gamma \in G$,
        \item
        $\sup_{x \in \tM} d(x, h(x)) < \infty$,
        \item
        $h$ and $h \inv$ are uniformly continuous.\qed
    \end{itemize}  \end{lemma}

To end the section, we use these properties to prove Global Product Structure,
a property of the foliations which does not immediately follow from the leaf
conjugacy.
%by presenting some further properties of the leaf conjugacy as well as some other topological properties of the strong foliations which do not follow directly from the leaf conjugacy.
%
%This allows us to prove a property of the foliations which does not
%immediately follow from the leaf conjugacy.

\begin{prop} \label{solvgps}
    $f$ has Global Product Structure.
    That is, for $x,y \in \tM$
    \begin{enumerate}
        \item
        $\Fcs(x)$ intersects $\Wu(y)$ in a unique point,
        \item
        $\Fcu(x)$ intersects $\Ws(y)$ in a unique point,
        \item
        if $x \in \Fcs(y)$, then
        $\Fc(x)$ intersects $\Ws(y)$ in a unique point,
        \item
        if $x \in \Fcu(y)$, then
        $\Fc(x)$ intersects $\Wu(y)$ in a unique point.
    \end{enumerate}  \end{prop}
\begin{proof}
    The uniqueness of the intersections follows from Proposition
    \ref{csuint} (and its analogous statement for $\Fcu$ and $\Ws$).
    Therefore, it is enough to prove existence.
    As each leaf of $\Fcs$ intersects each leaf of $\Fcu$ in
    a single leaf of $\Fc$,
    items (1) and (2) will follow from (3) and (4).

    To prove existence of the intersection in (3), suppose $L \in \Fc$, $L
    \subset \Lcs \in \Fcs$ and $x \in \Lcs$.
    Then $h(f^n(x))$ is on the leaf $\Ac(h(x))$ for all $n$,
    and by Lemmas \ref{plip}, \ref{toinfty}, and \ref{niceh}
    \[
        \lim_{n \to \infty} p_1(h(f^n(x))) = +\infty.
    \]
    By the definition of $\Ac$ and $\Acs$, $\lim_{n \to \infty} \dist(h(f^n(x)), \cL)
    = 0$ for any leaf $\cL \in \Ac$ where $\cL \subset \Acs(h(x))$.
    In particular, if $\cL = h(L)$, then by the uniform continuity of $h \inv$,
    $\lim_{n \to \infty} \dist(f^n(x), L) = 0$.
    Consequently, there is $n$ such that $f^n(x)$ is close enough to $L$ that
    $\Ws(f^n(x))$ intersects $L$.
    Applying $f^{-n}$ shows that $\Ws(x)$ intersects $L$ and (3) is proved.
    The proof of (4) is analogous.
\end{proof}

%%%%%%%%%%%%%%%%%%%%%%%%%%%%%%%%%%%%%%%%%%%%%%%%%%%%%%%%%%%%%%%%%%%%%%5
\appendix

\section{Quotients of partially hyperbolic diffeomorphisms}\label{Apendix-FiniteQuotients}

This appendix studies the algebraic properties of partially hyperbolic systems
on solvmanifolds
and is divided into three subsections.
The first classifies partially hyperbolic systems on suspension manifolds when
they are not homotopic to the identity and gives the proof of Proposition
\ref{proplist3}.
The second subsection proves Proposition \ref{proplist2}.
The last subsection classifies partial hyperbolicity on quotients of the
3-torus and other nilmanifolds.

\subsection{Algebra and leaf conjugacies} \label{sec-solvalg} %{{{1

Since $\tilde M_A$ is homeomorphic to $\RR^3$, if a diffeomorphism on $M_A$ is not homotopic to the identity, then it will induce
a non-trivial isomorphism of $G = \pi_1(M_A)$.  We first study these
isomorphisms and their corresponding algebraic maps on $\tM$.

Recall a point of $\tM$ can be written as $(x,t)$ with $x \in \bbR^2$ and $t \in
\bbR$.
For $z \in \bbZ^2$, let $\alpha_z$ denote the map
$\alpha_z(x,t) = (x+z,t).$
Then,
$G_0 = \{ \alpha_z : z \in \bbZ^2 \}$ is a normal subgroup of $G$, and
$G$ is generated by elements of $G_0$ together with the deck transformation
$\gam_3(x,t) = (A \inv x, t+1)$.

\begin{lemma} \label{autG}
    For any isomorphism $\phi:G \to G$ there are $B \in GL(2,\bbZ)$,
    $v \in \bbZ^2$, and $e \in \{+1,-1\}$ such that
    $\phi(\alpha_z) = \alpha_{B z}$ for $z \in \bbZ$ and
    $\phi(\gamma_3) = \alpha_v \gamma_3^e.$
    Moreover, $A^e B = B A$ and $B^j = \pm A^k$ for some $j,k \in \bbZ$ with
    $j > 0$.
\end{lemma}

See the appendix of \cite{ghys1980stabilite} for a proof of the
main statement of the lemma.
The fact that $B^j = \pm A^k$ is a consequence of the following classical
result.

\begin{lemma} \label{baake}
    If $A \in GL(2,\bbZ)$ is hyperbolic, then there is $A_0 \in GL(2,\bbZ)$ such
    that $B A = A B$ if and only if $B$ is of the form $\pm A_0^k$ for some
    $k \in \bbZ$.
\end{lemma}

See, for instance, \cite{baake1997} for a proof.

Due to Lemma \ref{autG}, every automorphism of $G$ is induced by an algebraic
map on $\tM$.

\begin{lemma} \label{algmap}
    For any automorphism $\phi:G \to G$ there are unique $B \in GL(2,\bbZ)$,
    $w \in \bbR^2$, and $e \in \{+1,-1\}$ such that the diffeomorphism
    \[
        \Phi:\tM \to \tM, \quad (x,t) \mapsto (Bx+w, e t)
    \]
    satisfies $\Phi \circ \gamma = \phi(\gamma) \circ \Phi$ for all $\gamma \in G$.
    As such, $\Phi$ quotients to a map $\Phi_0:M_A \to M_A$.
    If $e=+1$, then $\Phi(\Acs)=\Acs$ and $\Phi(\Acu)=\Acu$.
    If $e=-1$, then $\Phi(\Acs)=\Acu$ and $\Phi(\Acu)=\Acs$.
\end{lemma}
\begin{proof}
    Let $B, v, e$ be as in Lemma \ref{autG}.
    Let $w \in \bbR^2$ solve the equation $A^{-e} w + v = w$.
    Such a solution exists and is unique as $A$ is hyperbolic.
    Then, one can verify directly that
    $\Phi \circ \gamma = \phi(\gamma) \circ \Phi$.
    If $A B = B A$, then $B$ preserves the eigenspaces of $A$.
    If $A \inv B = B A$, then $B$ maps an eigenspace $E_\lam$ of $A$ to the
    eigenspace $E_{\lam \inv}$.
    From this, the claims about $\Phi(\Acs)$ and $\Phi(\Acu)$ follow.
\end{proof}
\begin{lemma} \label{nicealgmap}
    For any such $\Phi$, there is an iterate $n > 0$ which is of the form
    \[
        \Phi^n(x,t) = (A^m x + z, t)
    \]
    where $m \in \bbZ$ and $z \in \bbZ^2$.
\end{lemma}
\begin{proof}
    The equation $A^{-e} w + v = w$ from the last proof can be written as a
    system of linear equations with integer coefficients and can therefore be
    solved over $\bbQ$.
    This means that $x \mapsto Bx+w$ defines a permutation of
    a finite set of the form $\bbZ/q \times \bbZ/q$ for some $q$.
    As some iterate $\ell$ of this permutation fixes $(0,0)$ the corresponding
    function on $\bbR^2$ is of the form $x \mapsto B^\ell x + z$ for some $z \in
    \bbZ^2$.
    Suppose $B^j = \pm A^k$.
    Then, $\Phi^n$ for $n = 2 j \ell$ will have the desired form.
\end{proof}
\begin{cor} \label{algcor}
    If $\Phi_0:M_A \to M_A$ is an algebraic map as in Lemma \ref{algmap}, then there
    is $n  \ge  1$ such that $\Phi_0^n$ is homotopic to the identity.
\end{cor}
\begin{proof}
    One may write $M_A$ as the quotient $\bbT^2 \times \bbR /{\sim}$ where
    $(Ax, t) \sim (x, t+1)$.
    Then by Lemma \ref{nicealgmap}, $\Phi_0^n:M_A \to M_A$ can be written as
    $\Phi_0^n(x,t) = (A^m x + z, t) = (x, t+m)$
    which is clearly homotopic to the identity.
\end{proof}
From this, we can now prove Proposition \ref{proplist3}.

\begin{proof}
    [Proof of Proposition \ref{proplist3}]
    If $f:\tM \to \tM$ descends to $f_0:M_A \to M_A$, then by Lemma \ref{algmap},
    there is $\Phi$ descending to $\Phi_0$ such that $f$ and $\Phi$ induce the
    same automorphism of $\pi_1(M_A)$.  As $M_A$ is $K(\pi,1)$, Whitehead's
    theorem implies that $f_0$ and $\Phi_0$ are homotopic.
    By Corollary \ref{algcor}, there is $n$ such that $f_0^n$ is homotopic to
    the identity.
\end{proof}
If the $B$ in Lemma \ref{algmap} is hyperbolic, the corresponding $\Phi$ is
partially hyperbolic.  If $B = \pm Id$, one can compose $\Phi$ with a flow along
the center foliation $\Ac$ to produce a partially hyperbolic map.
In any case, we say a diffeomorphism $g:\tM \to \tM$
which preserves foliations $\Fcu,\Fcs,\Fc$ is leaf conjugate to $\Phi$
if there is a homeomorphism $h:\tM \to \tM$ such that $h(\Fsig)=\Asig$ and
\[
    h(f(L))=\Phi(h(L))
\]
for all $L \in \Fsig$ and $\sigma=c,cs,cu$.  If $g,\Phi,h$ quotient to maps
$g_0,\Phi_0,h_0$ on $M_A$, we say $g_0$ and $\Phi_0$ are leaf conjugate as well.

\begin{teo} \label{solvconj}
    Suppose $g_0:M_A \to M_A$ is partially hyperbolic and there is no periodic
    torus tangent to $\Ecs$ or $\Ecu$.
    Then $g_0$ is dynamically coherent and is leaf conjugate to an algebraic
    map $\Phi_0$ as given in Lemma \ref{algmap} with $e=+1$.
\end{teo}
\begin{proof}
    Dynamical coherence follows from Proposition \ref{proplist4}, so we need
    only establish the leaf conjugacy.
    Choose a lift $g:\tM \to \tM$ of $g_0$.
    This defines an isomorphism $g_*=\phi:G \to G$ by
    \[
        g \circ \gamma = \phi(\gamma) \circ g.
    \]
    Let $\Phi$ be the corresponding algebraic map given by Lemma \ref{algmap}.
    Note that the function $x \mapsto d(g(x), \Phi(x))$ quotients down to $M_A$
    and is therefore bounded.
    This shows
    \[
        \sup_{x \in \tM} d(g(x), \Phi(x)) < \infty.
    \]
    By Lemma \ref{nicealgmap}, there is $n$ such that
    $\Phi^n(v,t) = (A^k v + z, t)$
    for some $k$ and $z$.
    Define $\hat \gam = \gam_3^k \circ \alpha_z$, $f = \hat \gam \circ g^n$, and
    $\Phi_f = \hat \gam \circ \Phi^n$.
    One can check from the definitions of $\alpha_z$ and $\gam_3$ that
    $\Phi_f(v,t) = (v, t+k)$.
    From the definition of the metric on $\tM$, $\sup_x d(\Phi_f(x), x) < \infty$.
    In fact, the distance from $(v,t)$ to $(v,t+k)$ is exactly $k$.
    As $\hat \gam$ is an isometry, $\sup_x d(f(x), \Phi_f(x)) < \infty$ and so
    $\sup_x d(f(x), x) < \infty$ as well.
    Thus $f$ satisfies the assumptions in Section \ref{Sect-LeafConj}.

    Let $h:\tM \to \tM$ be the homeomorphism given by Lemma \ref{niceh}.
    Since $h$ is a finite distance from the identity,
    it follows that $\sup_x \dist(h(g(x)), \Phi(h(x))) < \infty$.
    For any leaf $L \in \Fcs$, $h(L)$ is a leaf of $\Acs$ and by Lemma
    \ref{uniqueCL},
    \[
        \HD(h(g(L)), \Phi(h(L))) < \infty \spacearrow h(g(L)) = \Phi(h(L)).
    \]
    The same reasoning applies to $\Fcu$ and shows that $h$ is a leaf conjugacy
    between $g$ and $\Phi$.
    As $g$ maps $\Fcs$ to $\Fcs$, it follows from the leaf conjugacy that $\Phi$
    maps $\Acs$ to $\Acs$, and therefore $e=1$ in the formula given in Lemma
    \ref{algmap}.
    As all of the maps descend to $M_A$, the result is proved.
\end{proof}
\subsection{Finite quotients of suspension manifolds} \label{sec-quotients} %{{{1

\begin{proof}
    [Proof of Proposition \ref{proplist2}.]
    Suppose $M$ is finitely covered by $M_A$.
    Without loss of generality (see the discussion at the end of Section
    \ref{SectionReductionSolvCase}), assume that this is a normal covering.
    By Lemma \ref{canlift}, an iterate of the partially hyperbolic
    diffeomorphism $f:M \to M$ lifts to $g:M_A \to M_A$ and, by Proposition
    \ref{proplist3}, we may assume that this lift $g$ is homotopic
    to the identity.

    If there is a $g$-periodic torus tangent to $\Ecs$ or $\Ecu$ in $M_A$, it
    would quotient to an $f$-periodic torus on $M$ and
    the results in
    \cite{HHU3} would imply that $M$ is a suspension manifold.
    Therefore, we may freely assume there is no such torus.

    Propositions \ref{proplist4} and \ref{proplist5}
    imply that $g$ is leaf conjugate to a suspension Anosov flow.
    As a consequence of Lemma \ref{projc}, there is a center
    foliation $\cF$ on $M_A$ and a map $p:M_A \to \bbR/\bbZ$, such that
    restricted to each leaf of $\cF$, $p$ is a $C^1$ covering.

    Let $H$ be the finite group of deck transformations $\tau:M_A \to M_A$
    associated to the covering $M_A \to M$.
    By Proposition \ref{liftuniq}, $\tau(\cF)=\cF$ for all $\tau \in H$.
    The space of leaves of $\cF$ is a non-Hausdorff space homeomorphic to
    $\bbT^2/A$.
    By the Lefschetz fixed point
    formula, for a homeomorphism $h:\bbT^2 \to \bbT^2$, either $h$ has a fixed
    point or $A \circ h$ has a fixed point (or both).  This implies that a deck
    transformation $\tau \in H$ fixes at least one leaf $L \in \cF$.
    If $\tau$ reversed the orientation of leaves, it would have a fixed point
    on $L$, which is impossible for a non-trivial deck transformation.  Thus,
    all $\tau \in H$ preserve the orientation of $\cF$.

    Define a map $\hat p : M_A \to \bbR/\bbZ$ by
    \[
        \hat p(x) = \sum_{\tau \in H} p(\tau(x)).
    \]
    Then, $\hat p \circ \tau = \hat p$ for all $\tau \in H$, the set $\hat p
    \inv(\{0\})$
    consists of $n$ disjoint tori where $n = |H|$, and each $\tau$ permutes
    these tori.
    Choose one torus and call it $S$.
    Starting at a point $x \in S$, flow forward along the leaf $\cF(x)$ until it
    intersects a torus of the form $\tau_1 S$ at a point $y$.
    Define $\psi(x)=\tau_1 \inv(y)$.  This defines a continuous map $\psi:S \to S$
    and the choice of $\tau_1$ is independent of $x$.

    Define $H_0 := \{ \tau \in H : \tau(S) = S \}$.
    Then, $\psi$ quotients to a function
    $\hat \psi: S / H_0 \to S / H_0$, and
    one can verify that $M$ is homeomorphic to the manifold defined as the
    suspension of $\hat \psi$.
    As $\psi^n$ is topologically conjugate to a power of $A$, $\hat \psi$ is
    expansive and therefore topologically conjugate to a hyperbolic
    toral automorphism \cite{Lew}.

    This shows that $M$ is homeomorphic to a suspension manifold.  As these
    manifolds are three-dimensional, they are diffeomorphic as well \cite{moise}.
\end{proof}

\subsection{Quotients of the torus and nilmanifolds} %{{{1

\begin{prop} \label{torushyp}
    Suppose $f$ is a homeomorphism on $\bbT^3$ such that the induced
    automorphism on $H_1(\bbT^3,\bbR)$ is hyperbolic.
    If $f$ descends to a homeomorphism of a finite quotient $\bbT^3 / \Gamma$,
    then this quotient is homeomorphic to the 3-torus.
\end{prop}
\begin{proof}
    Suppose $\bbT^3 / \Gamma$ is such a quotient.
    Then $\Gamma$ is a finite group of fixed-point free homeomorphisms of
    $\bbT^3$.
    As explained in \cite{ha2002classification}, we may freely assume that each
    $\gam \in \Gamma$ is affine.  That is, $\gamma$ is of
    the form $x \mapsto A x + b$ where $A$ is a toral automorphism and
    $b \in \bbT^3$.
    As $f$ descends to the quotient, the map $\Gamma \to \Gamma$,
    $\gam \mapsto f \gam f \inv$ is a well-defined automorphism.
    As $\Gamma$ is finite, there is $n$ such that $\gam = f^n \gam f^{-n}$
    for all $\gam \in \Gamma$. Assume $n=1$.

    The Lefschetz number of a homeomorphism of the 3-torus is
    \[
        (1 - \lam_1)(1 - \lam_2)(1 - \lam_3)
    \]
    where the $\lam_i$ are the eigenvalues of the induced automorphism of
    $H_1(\bbT^3, \bbR)$ \cite{Manning}.
    Consider $\gam \in \Gamma$. As $\gam$ is fixed-point free, one of these
    eigenvalues must be equal to one.
    The associated eigenspace is given by the solution of a system of
    linear equations with integer coefficients.
    As this can be solved over $\bbQ$, the first homology group with
    coefficients in $\bbZ$ has a subgroup defined by
    \[
        E = \{ v \in H_1(\bbT^3, \bbZ) : \gam_* v = v \}
    \]
    which is non-empty.

    Since $f$ and $\gam$ commute, $f_*$ on $H_1(\bbT^3,\bbZ)$
    restricts to an automorphism of $E$.
    If $E$ is rank 1, the automorphism shows that $f_*$ has an eigenvalue of
    $\pm 1$.
    If $E$ is rank 2, the automorphism shows that $f_*$ has two eigenvalues
    whose product is $\pm 1$, and the third eigenvalue is therefore $\pm 1$.
    Neither case is possible if $f_*$ is hyperbolic.
    Hence, $E$ has full rank and by its definition $E = H_1(\bbT^3, \bbZ)$.
    Writing $\gam$ as an affine map $x \mapsto A x + b$, it must be that
    $A$ is the identity.
    Thus, $\Gamma$ is a group of translations, and
    $\bbT^3/\Gamma$ is homeomorphic to the 3-torus.
\end{proof}
\begin{prop} \label{toruscover}
    Suppose $f$ is a homeomorphism on $\bbT^3$ such that
    the induced automorphism on $H_1(\bbT^3, \bbR)$ has eigenvalues $\lam_1,
    \lam_2, \lam_3$ satisfying
    \[
        0 < |\lam_1| < |\lam_2| = 1 < |\lam_3|.
    \]
    If $f$ descends to a homeomorphism of a finite quotient $\bbT^3 / \Gamma$,
    then this quotient is either homeomorphic to the 3-torus
    or double covered by the 3-torus.
\end{prop}
\begin{proof}
    As in the last proof, after replacing $f$ with an iterate, we may freely
    assume that $f \gam = \gam f$ for all $\gam$ and that each $\gam$ is an affine
    map on $\bbT^3$.
    Also, assume that the eigenvalues $\lam_i$ are all positive.

    Define $\Gam_0$ as the set of $\gam \in \Gam$ which are translations
    $x \mapsto x + b$.
    Then, $\Gam_0$ is a normal subgroup, $f$ descends to $\bbT^3/\Gam_0$
    and the quotient is homeomorphic to $\bbT^3$.
    Therefore, up to replacing $\bbT^3$ by $\bbT^3/\Gam_0$, we may freely assume
    that no non-trivial $\gam \in \Gam$ is a translation.
    That is, $\gam_*  \ne  Id$ where $\gam_*$
    is the induced automorphism of $H_1(\bbT^3, \bbZ)$.

    Now suppose there is a non-trivial element $\gam \in \Gam$.
    Our goal is to show that $\gam$ is the unique non-trivial element.

    As $\lam_2 = 1$ is an eigenvalue, arguing as in the last proof, there
    is $v \in H_1(\bbT^3, \bbZ)$ such that $f_*v = v$.
    Choosing an appropriate basis for $H_1(\bbT^3, \bbZ) \cong \bbZ^3$,
    write $f_*$ as a $3 \times 3$ matrix
    \[
        f_* =
        \left(
        \begin{array}{c|c}
        A & {\begin{matrix}0\\0\end{matrix}} \\ \hline
        {\begin{matrix}*&*\end{matrix}}& 1
        \end{array}
        \right)
    \]
    where $A \in GL(2,\bbZ)$ is hyperbolic.
    As $\gam_*$ commutes with $f_*$,
    \[
        \gam_* =
        \left(
        \begin{array}{c|c}
        B & {\begin{matrix}0\\0\end{matrix}} \\ \hline
        {\begin{matrix}*&*\end{matrix}}& c
        \end{array}
        \right)
    \]
    where $A B = B A$.
    The upper-right entries are zero because that is the only
    $2 \times 1$ matrix $X$ which satisfies $A X = X$.
    Since $\Gam$ is finite, there is $n$ such that
    $\gam_*^n = Id$.  Consequently, $c = \pm 1$ and, by
    Lemma \ref{baake}, $B = \pm Id$.

    We now consider the four cases for $B$ and $c$.

    {\bf Case 1.}\ \ $B=+Id$ and $c=+1$.\\
    Here, $\gam_*$ is a lower triangular matrix with ones on the diagonal
    and $\gam_*^n = Id$ for some $n$.
    This implies that $\gam_* = Id$, which is not possible by an earlier
    assumption.
    Therefore, this case cannot occur.

    {\bf Case 2.}\ \ $B=-Id$ and $c=-1$.\\
    In this case, $\gam_*$ has $-1$ as the sole eigenvalue.  Using the
    Lefschetz fixed point theorem as in the
    previous proof, $\gam$ would have a fixed point, a contradiction.

    {\bf Case 3.}\ \ $B=+Id$ and $c=-1$.\\
    Suppose $\gam'$ is another non-trivial element of $\Gam$ where the matrix
    $\gam'_*$ has corresponding submatrices $B'$ and $c'$.
    If $B' = -Id$ and $c' = +1$,
    then $B B' = -Id$, $c c' = -1$, and the product $\gam \gam' \in \Gam$ would
    have a matrix of the form already ruled out in case 2 above.
    Therefore, $B' = +Id$, $c=-1$,
    and the product $\gam \gam'$ is of the form considered in case 1.
    This is only possible if $\gam \gam'$ is the identity element of $\Gam$.
    Since this holds for any non-trivial $\gam' \in \Gam$, it follows that $\Gam$
    is the two element group $\{Id, \gam\}$.

    {\bf Case 4.}\ \ $B=-Id$ and $c=+1$.\\
    This case is nearly identical to case 3 and is left to the reader.
\end{proof}
Not including the torus itself, there are exactly three manifolds (up to
either homeomorphism or diffeomorphism)
double covered by the 3-torus \cite{lee1993}.
They are defined by quotienting $\bbT^3 = \bbR^3/\bbZ^3$ by one of the following
maps
\begin{align*}
    \tau_1(x,y,z) &= (-x,\ -y,\ z + \tfrac{1}{2}) \\
    \tau_2(x,y,z) &= (x + \tfrac{1}{2},\ y,\ -z) \\
    \tau_3(x,y,z) &= (x+z + \tfrac{1}{2},\ y+z,\ -z).
\end{align*}
All three resulting manifolds admit partially hyperbolic
diffeomorphisms, as demonstrated by the toral automorphisms generated by the
matrices
\[
        \begin{pmatrix}
        2 & 1 & 0 \\
        1 & 1 & 0 \\
        0 & 0 & 1  \end{pmatrix}
    ,\quad
        \begin{pmatrix}
        3 & 1 & 0 \\
        2 & 1 & 0 \\
        0 & 0 & 1  \end{pmatrix}
    ,\quad \text{and} \quad
        \begin{pmatrix}
        5 & 2 & 3 \\
        2 & 1 & 1 \\
        0 & 0 & 1  \end{pmatrix}
    .
\]
Further, the classification up to leaf conjugacy can be extended from the
3-torus to its finite quotients.  To do this, we first state a consequence of
the classification on the 3-torus.

\begin{prop} \label{conjfromfund}
    Suppose $f_T,g_T:\bbT^3 \to \bbT^3$ are partially hyperbolic, dynamically coherent
    diffeomorphisms, with lifts $f,g:\bbR^3 \to \bbR^3$ such that
    the induced automorphisms $f_*$ and $g_*$ on $\pi_1(\bbT^3)$ are equal.
    Then, there is a leaf conjugacy $h:\bbR^3 \to \bbR^3$ between $f$ and $g$
    which descends to a leaf conjugacy $h_T:\bbT^3 \to \bbT^3$ between
    $f_T$ and $g_T$ such that $h_*$ is the identity.

    Further, $h$ is unique up to sliding along center leaves.  That is, if
    $h'$ is another such leaf conjugacy, then there is $C > 0$ such that for
    all $x \in \bbR^3$, $h(x)$ and $h'(x)$ lie on the same center leaf and
    $d_c(h(x), h'(x)) < C$.
\end{prop}
\begin{proof}
    As shown in \cite{Hammerlindl} and extended to the pointwise case in
    \cite{HP}, if $f$ and $g$ have the same action on $\pi_1(\bbT^3)$, they are
    both leaf conjugate to the same linear map on $\bbR^3$.  The stated results
    then follow from this and the properties proven in \cite{Hammerlindl}.
\end{proof}
Let $\Aff(\bbR^n)$ denote the group of affine maps on $\bbR^n$.

\begin{prop} \label{virtabel}
    Suppose that $f_M$ is a partially hyperbolic diffeomorphism on a 3-manifold $M$
    with virtually abelian fundamental group and that there is no 2-torus tangent
    to $\Ec \oplus \Eu$ or $\Ec \oplus \Es$.
    Then, there is a subgroup $\Gamma < \Aff(\bbR^3)$, and a partially hyperbolic
    linear map $A: \bbR^3 \to \bbR^3$
    such that $A$ descends to a map
    $A_M : \bbR^3/\Gam \to \bbR^3/\Gam$
    which is leaf conjugate to $f_M$.
\end{prop}
\begin{proof}
    As $M$ is virtually nilpotent, it is finitely covered by a circle bundle
    over a 2-torus \cite{Pw}.  To be virtually abelian, this circle bundle
    must be trivial. That is, $M$ is finitely covered by the 3-torus.

    If $M$ is diffeomorphic to the 3-torus, the proof follows from
    \cite{Hammerlindl} and \cite{HP}.

    Otherwise, by the results of this section, we may freely assume that $M =
    \bbR^3 / \Gamma$, where $\Gamma$ has an index-two subgroup $\Gamma_0$
    consisting of translations $x \mapsto x + z$ for $z \in \bbZ^3$.  As
    $\Gamma_0$ is the maximal nilpotent subgroup of $\Gamma$, it is preserved by
    every automorphism.
    Choose some element of $\Gamma \setminus \Gamma_0$ and call it $\alpha$.

    Lift $f_M$ to $f:\bbR^3 \to \bbR^3$.
    This defines an automorphism $f_*:\Gamma \to \Gamma$, and by the theorems of
    Bieberbach (see, for instance, \cite{lee-raymond}), there is an affine map
    $g:\bbR^3 \to \bbR^3$, $g(x) = A x + b$ such that $g$ descends to a
    diffeomorphism $g_M:M \to M$,
    and the induced automorphism $g_*:\Gamma \to \Gamma$ is equal to $f_*$.

    As $f$ is partially hyperbolic, its action on $\Gamma_0 \cong \bbZ^3$ is
    partially hyperbolic (\cite{BI}). Moreover, by \cite{Pot} (or \cite[Proposition 3.3]{HP}), it has eigenvalues $\lam_1,\lam_2,\lam_3$ such that
    \[
        |\lam_1|<|\lam_2|<|\lam_3| \quad \text{and} \quad |\lam_1| < 1 < |\lam_3|.
    \]
    By Proposition \ref{torushyp} and the assumption that $M$ is not the
    3-torus, it must be that $|\lam_2|=1$.
    Then, $g$ is partially hyperbolic and dynamically coherent.
    As the center bundle of $g$ is uniquely integrable, so is the center
    bundle of the quotiented map $g_M$ on $M$.
    This shows that the center foliation of $g$ is $\Gamma$-invariant.

    Let $h:\bbR^3 \to \bbR^3$ be the leaf conjugacy given by Proposition
    \ref{conjfromfund}.  That is, $hg(\cL) = fh(\cL)$ for every center leaf $\cL$
    of $g$, and $h \gam = \gam h$ for all $\gam \in \Gamma_0$.
    As $\Gamma_0$ is normal, $\gam \alpha h \alpha \inv = \alpha h \alpha \inv \gam$
    for all $\gam \in \Gamma_0$.
    Also, there is $\beta \in \Gamma_0$ such that
    $\alpha f \alpha \inv = f \beta$.
    Since, $f_*=g_*$, it follows that
    $\alpha g \alpha \inv = g \beta$,
    using the same $\beta$.
    Then,
    \begin{align*}
        fh(\cL) &= hg(\cL)  \quad \Rightarrow \\
        \alpha f \alpha \inv \alpha h \alpha \inv (\cL) &=
        \alpha h \alpha \inv \alpha g \alpha \inv (\cL)  \quad \Rightarrow \\
        f \beta \alpha h \alpha \inv (\cL) &=
        \alpha h \alpha \inv g \beta (\cL)  \quad \Rightarrow \\
        f \alpha h \alpha \inv \beta (\cL) &=
        \alpha h \alpha \inv g \beta (\cL)  \quad \Rightarrow \\
        f \alpha h \alpha \inv (\cL) &=
        \alpha h \alpha \inv g (\cL),
    \end{align*}
    so $\alpha h \alpha \inv$ is also a leaf conjugacy.
    By the uniqueness given in Proposition \ref{conjfromfund},
    $h(\cL) = \alpha h \alpha \inv(\cL)$ for every center leaf $\cL$ of $g$.
    That is, $h$ and $\alpha h \alpha \inv$ are {\em $c$-equivalent}, as
    defined in section 6 of \cite{Hammerlindl}.
    That section explains how to average a finite number of $c$-equivalent
    leaf conjugacies to get a new leaf conjugacy.  If we define $h_1$ as such
    an average of $h$ and $\alpha h \alpha \inv$, then one can verify that
    $h_1$ is a leaf conjugacy on $\bbR^3$ which quotients to $M = \bbR^3/\Gamma$.

    We have proven that $f_M$ is leaf conjugate to a quotient of an affine map
    $g:\bbR^3 \to \bbR^3,\ \ x \mapsto A x + b$, whereas the proposition claims
    that $f_M$ is leaf conjugate to a linear map.
    By conjugating $g$ with a translation $x \mapsto x + c$ for $c \in \bbR^3$,
    we can replace $g$ by the map $x \mapsto A x + A c - c + b$.
    Hence, we can assume that $b$ is in the null space of $A - Id$.
    This conjugation means that $\Gamma < \Aff(\bbR^3)$ will also be replaced by
    a conjugate subgroup of $\Aff(\bbR^3)$.
    Then, $(A - Id)b = 0$ implies that $b$ and the origin of $\bbR^3$ lie in
    the same center leaf of $g$.
    This shows that $g$ is leaf conjugate to the
    linear map $A$ and this leaf conjugacy (the identity on $\bbR^3$) descends
    to the quotient $\bbR^3 / \Gamma$, completing the proof.
\end{proof}
Similar results hold for finite quotients of 3-dimensional nilmanifolds.
Let $\Heis$ denote the Heisenberg group, consisting of all matrices of the form
\[
        \begin{pmatrix}
        1 & x & z \\
        0 & 1 & y \\
        0 & 0 & 1  \end{pmatrix}
\]
under multiplication.
For any integer $k  \ge  1$, the set of matrices with $x,y \in \bbZ$ and
$z \in \frac{1}{k}\bbZ$ defines a discrete subgroup $\Gamma_k$, and
the quotient $\Heis/\Gamma_k$ is a compact manifold, a \emph{nilmanifold},
denoted by $\cN_k$.  For each $\cN_k$, the homology group $H_1(\cN_k, \bbR)$ is
two-dimensional.
In fact, the projection
\[
    \Heis \to \bbR^2, \quad
        \begin{pmatrix}
        1 & x & z \\
        0 & 1 & y \\
        0 & 0 & 1  \end{pmatrix}
    \mapsto (x,y).
\]
defines a projection $\cN_k \to \bbT^2$, and this projection induces the isomorphism
between $H_1(\cN_k, \bbR)$ and $H_1(\bbT^2, \bbR)$.
An affine map on $\Heis$ is a map of the form $x \mapsto \Phi(x) \cdot c$
where $\Phi:\Heis \to \Heis$ is a Lie group automorphism and $c \in \Heis$.
An affine map on $\cN_k$ is a quotient of an affine map on $\Heis$.

\begin{prop}
    Suppose $f$ is a homeomorphism of $\cN_k$ such that the induced
    automorphism on $H_1(\cN_k, \bbR)$ is hyperbolic.
    If $f$ descends to a homeomorphism of a finite quotient $\cN_k / \Gamma$,
    then this quotient is either homeomorphic to $\cN_\ell$
    or double covered by $\cN_\ell$ for some $\ell$.
      \end{prop}
\begin{proof}
    As in the proofs of Propositions \ref{torushyp} and \ref{toruscover},
    we may assume the elements of $\Gamma$ are affine maps.
    (See
    \cite{lee-raymond} and \cite{choi2005}.)
    As explained in Section 2 of \cite{HNil}, to every homeomorphism of
    $\cN_k$ is associated an automorphism of the Lie algebra $\heis$, and this
    automorphism can be written as a $3 \times 3$ matrix
    \[
        \left(
        \begin{array}{c|c}
        A & {\begin{matrix}0\\0\end{matrix}} \\ \hline
        {\begin{matrix}*&*\end{matrix}}& \det(A)
        \end{array}
        \right).
    \]
    In fact, $A$ has integer entries and it is the matrix given by the
    automorphism induced by $f$ on $H_1(\cN_k, \bbR)$ (see the discussion after
    Proposition 5.1 in \cite{HNil}).

    Using these properties, the proof follows in the same manner as in
    Proposition \ref{toruscover}.
\end{proof}
The finite quotients of 3-nilmanifolds have been classified \cite{dekimpe1995}.
If $M$ is double covered by $\cN_k$, but is not itself a nilmanifold, then
$k$ is even and $M$ is diffeomorphic to $\Heis/\langle \Gamma_k, \tau_k \rangle$
where
\[
    \tau_k : \Heis \to \Heis, \quad
        \begin{pmatrix}
        1 & x & z \\
        0 & 1 & y \\
        0 & 0 & 1  \end{pmatrix}
    \mapsto
        \begin{pmatrix}
        1 & -x & z+\tfrac{1}{2k} \\
        0 & 1 & -y \\
        0 & 0 & 1  \end{pmatrix}
    .
\]
To see this, look at the list of quotients given in \cite{dekimpe1995}
and note that only item 2 of the list is a double cover.

Every such manifold supports a partially hyperbolic diffeomorphism, as
evidenced by the map
\[
    \Heis \to \Heis, \quad
        \begin{pmatrix}
        1 & x & z \\
        0 & 1 & y \\
        0 & 0 & 1  \end{pmatrix}
    \mapsto
        \begin{pmatrix}
        1 & 5x+2y & z+5x^2+y^2+4x y \\
        0 & 1 & 2y+z \\
        0 & 0 & 1  \end{pmatrix}
    .
\]
This map also shows that $\cN_k$ supports a partially hyperbolic
diffeomorphism for every $k$ (even or odd).

\begin{prop}
    Suppose $f_M$ is a partially hyperbolic diffeomorphism on a 3-manifold $M$
    with a fundamental group which is virtually nilpotent and not virtually
    abelian.
    Then, there is a subgroup $\Gamma < \Aff(\Heis)$, and a partially
    hyperbolic automorphism $\Phi: \Heis \to \Heis$
    such that $\Phi$ descends to a map
    $\Phi_M: \Heis/\Gamma \to \Heis/\Gamma$
    which is leaf conjugate to $f_M$.
\end{prop}
\begin{proof}
    As shown in \cite{Pw}, $M$ is finitely covered by a circle
    bundle over the torus.  As the fundamental group is not virtually
    abelian, the circle bundle is non-trivial and
    such non-trivial circle bundles are exactly the
    nilmanifolds $\cN_k$.
    Then, $f_M$ is dynamically coherent by \cite[Theorem 1.2]{HP}.

    The proof in this case now follows almost exactly as the proof of
    Proposition \ref{virtabel}, but with $\bbR^3$ replaced by $\Heis$.  The
    relevant theorem of Bieberbach used to construct an affine map $g$ also
    holds in the case of nilmanifolds, as proven by Lee and Raymond
    \cite{lee-raymond}.
    Partial hyperbolicity of $g$ with $|\lam_1| < |\lam_2| = 1 < |\lam_3|$
    is proven in \cite{HNil}.  \end{proof}
%

%%%%%%%%%%%%%%%%%%%%%%%%%%%%%%%%%%%%%%%%%%%%%%%%%%%%%%%%%%%%%%%%%%%%%%%%%%%%%%%%%%%%%%%%%%%%%%%%%%%%%%%%%%%%%%%%%%%%%%%%%
\section{On center-stable tori}\label{Appendix-AnosovTori} %{{{1

This appendix treats partially hyperbolic diffeomorphisms admitting
center-stable or center-unstable tori. The first subsection proves Proposition
\ref{propcompactleaf}.
In the second subsection the case of
absolute partial hyperbolicity is discussed: it is shown that under this more
restrictive definition, the existence of such tori is not possible. This
allows to recover the main result of \cite{BBI2} without (explicit) use of
quasi-isometry and extend it to the case of suspension manifolds.

\subsection{Proof of Proposition \ref{propcompactleaf}} %{{{1

\begin{proof}[Proof of Proposition \ref{propcompactleaf}]
Assume that $\cF^{cs}$ has a leaf $L$ which projects down to a compact surface
$T$ in $M$. As $T$ is tangent to $E^{cs}$ it admits a foliation without circle
leaves and thus must be a torus.
Since $L$ is homeomorphic to a plane, the torus $T$ is incompressible. If
$L_1$ and $L_2$ are leaves of $\cF^{cs}$ which project to tori, item (5)
of Theorem \ref{thmbran} shows their
projections are disjoint modulo isotopy.
By classical results in 3-manifold topology, there are finitely many disjoint
incompressible tori modulo isotopy, so by replacing the diffeomorphism $f$ on
$M$ by an iterate, assume $f(T)$ is isotopic to $T$.

Take $x \in T$ and consider a sequence $n_k$ such that $x_{k}=f^{-n_k}(x)$
converges.
Lift $x_k$ to a convergent sequence $\tilde x_k$ in the cover $\tilde M$ and
lift each
torus $T_k = f^{-n_k}(T)$ to a leaf $L_k$ in $\Fcs$ through $\tilde x_k$.
As shown by \cite[Lemma 7.1]{BI}, these leaves $L_k$ converge in the
$C^1$ topology to a leaf $L_\infty$ in $\Fcs$.
Since there is a subgroup $\Gamma$ isomorphic to $\ZZ^2$ of deck
transformations which fix all the leaves $L_k$ this group $\Gamma$ also fixes
$L_\infty$.

Projecting down, the limit leaf
$T_\infty \subset M$ contains a copy of $\bbZ^2$ in its fundamental group and
is therefore compact (\cite{ClasificacionSuperficies}).
Arguing as above, $T_\infty$ is a torus.
Construct a normal neighborhood
$N$ of $T_\infty$ consisting of small unstable segments.
Then, there is an arbitrarily large iterate $f^{-\ell}$ which maps a torus
$T_{k_1}$ to $T_{k_2}$ where both tori are arbitrarily close to $T_\infty$.
Since these tori are transverse to the unstable bundle on $M$,
one can show that $f^{-\ell}(\overline N) \subset N$ and therefore there is a
normally repelling $f$-periodic torus tangent to $\Ecs$ inside $N$.
\end{proof}

\begin{cor}\label{Coro-TorusLeafImpliesfperiodicTorus} Assume $f: M \to M$ is
    a partially hyperbolic diffeomorphism such that the bundles $E^s, E^c,
    E^u$ are orientable then $M$ admits a codimension one foliation
    without compact leaves.
\end{cor}

\begin{proof}
    If $M$ admits an invariant torus tangent to $\Ecs$ or $\Ecu$, then
    it is one of
    the manifolds listed in \cite{HHU3}, all of which have admit foliations
    without compact leaves.  Otherwise, the result follows by Propositions
    \ref{propbi72} and \ref{propcompactleaf}.
\end{proof}

\subsection{Absolute partial hyperbolicity and center-stable tori} %{{{1

To end this appendix, we give a different proof of the main result of
\cite{BBI2}, which also applies to solvmanifolds.
A partially hyperbolic diffeomorphism $f: M \to M$ with
splitting $TM=E^s \oplus E^c \oplus E^u$ is \emph{absolutely partially
hyperbolic} if there exist constants $0<\gamma_1 < 1< \gamma_2$ and $N>0$ such
that for every $x\in M$:
\[ \|Df^N|_{E^s(x)}\| < \gamma_1 < \|Df^N|_{E^c(x)}\| < \gamma_2 <
\|Df^N|_{E^u(x)}\|. \]

\begin{teo}\label{Teo-Absolute}  Let $f: M \to M$ be a partially hyperbolic diffeomorphism admitting a two dimensional $f$-periodic torus $T$ tangent to $E^{cs}$. Then, $f$ is not absolutely partially hyperbolic.
\end{teo}

\begin{proof} Assume $T^{cs}$ is an $f$-invariant torus tangent to $E^{cs}$.
Then, the dynamics in $T^{cs}$ must be semiconjugated to a certain linear
Anosov diffeomorphism $A$ of $\TT^2$ (see \cite{HHU3}).
The entropy of $f|_{T^{cs}}$ is at least as big as the entropy of $A$. Using
the variational principle and Ruelle's inequality (see \cite{ManheLibro}), for
every $\eps>0$ there is an ergodic measure $\mu^{\eps}$ such that the center
Lyapunov exponent of $\mu^{\eps}$ is at least $h_{top}(A)-\eps$.

On the other hand, adapting the proof of Lemma 4.5 in \cite{HNil} shows that
the asymptotic rate of expansion, $\gam_2$, along the center is
strictly smaller than than $h_{top}(A)$, which is equal to the largest
Lyapunov exponent of $A$.
This gives a contradiction.
\end{proof}

As a consequence we obtain that every absolutely partially
hyperbolic diffeomorphism of a 3-manifold with (virtually)
solvable fundamental group is dynamically coherent.

%%%%%%%%%%%%%%%%%%%%%%%%%%%%%%%%%%%%%%%%%%%%%%%%%%%%%%%

\end{document}